\documentclass[12pt]{amsart}
\usepackage{mathrsfs}
\usepackage{amsfonts}
\usepackage{amsmath}
\usepackage{amssymb}
\usepackage{graphicx}
\usepackage{color}
\usepackage{xcolor}
\usepackage{amsthm,amscd}
\usepackage{calc}
\usepackage{hyperref}
\usepackage{setspace}

\newtheorem{thm}{Theorem}[section]
\newtheorem{cor}[thm]{Corollary}
\newtheorem{lem}[thm]{Lemma}
\newtheorem{prop}[thm]{Proposition}

\newcommand{\ra}{\rangle}
\newcommand{\la}{\langle}

\newcommand{\n}{\nabla}

\newcommand{\ib}{\bar{i}}
\newcommand{\jb}{\bar{j}}
\newcommand{\kb}{\bar{k}}
\newcommand{\lb}{\bar{l}}

\newcommand{\dbar}{\bar{\partial}}

\newcommand{\w}{\omega}

\newcommand{\vn}{\vec{n}}

\newcommand{\pd}{\partial}

\newtheorem{theorem}{Theorem}

\newtheorem{remark}[theorem]{Remark}

\title[]{Dimension estimate and existence of holomorphic sections with polynomial growth on gradient K\"ahler Ricci shrinkers}

\author{Fei He}
\author{Jianyu Ou}

\email{hefei@xmu.edu.cn and oujianyu@xmu.edu.cn}
\address{School of Mathematical Science, Xiamen University, 422 S. Siming Rd. Xiamen, Fujian, P.R.China, 361005.}

\begin{document}

\begin{abstract}
In this paper, we show that the dimension for the space of holomorphic $(p, 0)$-forms, and holomorphic sections of the pluri-anticanonical line bundle $K_M^{-q}$ with polynomial growth on gradient K\"ahler Ricci shrinkers with bounded curvature is bounded from above by a power function of the growth rate.  We also prove the existence of holomorphic sections of $K_M^{-q}$ with polynomial growth when the K\"ahler Ricci shrinker is asymptotically conical, provided $q$ is sufficiently large. As an application, we show that the Kodaira map constructed using such sections is a holomorphic embedding into a complex projective space.
\end{abstract}

\maketitle

\section{Introduction}
On a complete K\"ahler manifold, we say a holomorphic function $u$ has polynomial growth if there are constants $d$ and $C$, such that 
\[
|u(x)| \leq C(1+r(x))^d
\]
for all $x$, where $r(x)$ is the geodesic distance to a fixed point; we denote the linear space of such functions as $\mathcal{O}_d$. Holomorphic functions, and holomorphic sections of some bundles, are natural objects to study since they yield information about the complex structure of the manifold. There are extensive studies of holomorphic functions with polynomial growth on K\"ahler manifolds with nonnegative bisectional curvature or nonnegative Ricci curvature (see for example \cite{Mok1984}, \cite{Ni2004}, \cite{CFYZ06} \cite{Liu2016}, \cite{Liu2021a}, \cite{Liu2023} and references therein). Holomorphic sections of the pluri-anticanonical bundle with polynomial growth were studied in \cite{Mok1990} and \cite{Liu2021b}, and were applied to construct embeddings into projective spaces. 

Recall that a gradient shrinking Ricci soliton (also called Ricci shrinker) is a triple $(M, g, f)$, where $(M,g)$ is a complete Riemannian manifold, and $f$ is a smooth function satisfying the equation
\begin{equation}\label{eqn: soliton equation}
Ric + \n\n f = \frac{1}{2}g.
\end{equation}
The importance of gradient Ricci solitons lies in their role as self-similar solutions to the Ricci flow, and they can serve as models for the formation of singularities. Understanding the geometry of gradient Ricci solitons is therefore a central problem in the study of the Ricci flow, as it provides critical insights into the behavior of the flow near singular points. Recently, gradient K\"ahler Ricci shrinkers of complex dimension 2 are classified, see \cite{MW2019}\cite{CDS2024}\cite{CCD2024}\cite{BCCD}\cite{LW2023} and related works. The classification for higher dimension is widely open at the moment. 

The study of holomorphic functions and holomorphic forms with polynomial growth on gradient K\"ahler Ricci solitons was initiated by Munteanu and Wang \cite{MW2014}\cite{MW2015}. They proved that on a gradient K\"ahler Ricci shrinker, the linear space $\mathcal{O}_d$ is finite dimensional \cite{MW2014}.  
In this paper, we study holomorphic $(p,0)$-forms and sections of the pluri-anticanonical bundles, in hope to get useful information about the complex structure of the underlying manifold. In our previous work \cite{HO2024}, we showed that holomorphic functions with polynomial growth on the complete K\"ahler Ricci shrinkers can be written as finite linear combinations of holomorphic eigenfunctions of the drifted Laplacian $\Delta_f:=\Delta-\nabla_{\nabla f}$. Similar statements also hold for holomorphic $(p, 0)$-forms, for which the operator to look at is the weighted Hodge Laplacian $\Delta_f^d$. Therefore, it will be useful to study the eigenvalues and eigenfunctions of these operators. For generality, we will do this for a class of operators acting on sections of vector bundles. 

{{For the purposes of this article, the spectrum of an operator will always mean its $L^2$ spectrum. }}For an operator $L$ with discrete spectrum, we denote its eigenvalues as $\lambda_1 \leq \lambda_2 \leq \lambda_3 \leq ...$. Let $\mathcal{N}(\lambda)= \sup \{k:\lambda_k \leq \lambda \}$, this is called the \textbf{spectral counting number} in some of the literature. The study of the spectral counting number for elliptic operators has a long history which we shall not survey in this article, our main interest is to give an upper estimate of the spectral counting number for certain operators on gradient Ricci shrinkers.

Let $dv_f = e^{-f}dv_g$ be a weighted measure. Consider a vector bundle $E$ over $M$, equipped with an inner product $\langle, \rangle_E$ and a compatible connection $\n$. Let $L = - \Delta  + \n_{\n f}+ A$ be an operator acting on sections of $E$ , where $A$ is a symmetric endomorphism of $E$. Then $L$ is self-adjoint w.r.t. the weighted $L^2$ inner product $\int_M \la , \ra_E dv_f$. Suppose $A$ is bounded, then the same arguments as in \cite{HN2014} and \cite{CZ2016} shows that $L$ has discrete spectrum bounded from below. We have the following estimate for its spectral counting number.

\begin{thm}\label{thm: estimate of spectral counting number}
Suppose $(M^n, g, f)$ is a complete gradient shrinking Ricci soliton with bounded curvature, $E \to M$ is a vector bundle equipped with a metric $\la, \ra$ and a compatible connection. Let $L = - \Delta  + \n_{\n f}+ A$, where $\Delta$ is the rough Laplacian and $A\in \Gamma(End(E))$ satisfies $\la A(u), u\ra \leq \Lambda |u|^2$ for any $u \in \Gamma(E)$, suppose $\Lambda$ is taken s.t. $\lambda_1(L)+ \Lambda \geq \frac{1}{2}$. Then the spectral counting number of $L$ satisfies
\[
\mathcal{N}(\lambda) \leq C \hspace{2pt }rank(E) (1 +\lambda +\Lambda) ^{{n+1}},
\]
where $C$ is a constant depending on the geometry. 
\end{thm}

 As a consequence, we have the following dimension estimate.
For any integer $p\geq 0$, 
let $\mathcal{O}_{\mu}(\Lambda^{p,0})$ be the linear space of holomorphic $(p, 0)$-forms $u$ with polynomial growth of order at most  $\mu$, i.e. $|u(x)|\leq C(1+r(x))^\mu$ for some constant $C$, where $r(x)$ is the geodesic distance to some fixed point.  
\begin{cor}\label{cor: estimate of the dimension of holomorphic forms}
 Let $(M^{2m}, g, f)$ be a complete gradient K\"{a}hler Ricci shrinker with bounded curvature. Let $K = \sup_M|Ric| + 1$. Then for any $p\geq 0$, we have
\[
\dim \mathcal{O}_{\mu}(\Lambda^{p,0}) \leq C \hspace{2pt }rank(\Lambda^{p,0}) (1 +\mu + pK)^{ 2m+1},
\]
where $C$ is a constant depending on the geometry. 
\end{cor}

\begin{remark}
 By tracing the constants carefully in Munteanu and Wang's proof \cite{MW2014}, one can see that the dimension is bounded from above by an exponential of $d$. In the same paper, they remarked that it is possible to control the dimension explicitly, as a polynomial of $d$, but they did not give the proof.   Since $\mathcal{O}_d = \mathcal{O}_d(\Lambda^{0,0})$, Corollary \ref{cor: estimate of the dimension of holomorphic forms} confirms the claim by Munteanu and Wang with bounded curvature by {a} totally different method. Note that in \cite{HO2024}, we have a better estimate for $\dim \mathcal{O}_d$ under some stronger curvature assumptions. 
\end{remark}

Similarly, let $K_M^{-1}$ be the anti-canonical line bundle, let $q$ be a positive integer, and denote $\mathcal{O}_\mu(K_M^{-q})$ as the linear space of holomorphic sections of $K_M^{-q}$ with polynomial growth of order at most $\mu$, we have the following dimension estimate:
\begin{cor}\label{cor: dimension estimate for power of anticanonical bundle}
Let $(M^{2m}, g, f)$ be a complete gradient K\"ahler Ricci shrinker with bounded curvature. For any integer $q> 0$, we have 
\[
\dim \mathcal{O}_d(K_M^{-q}) \leq C (1 + q + d)^{2m+1 },
\]
where $C$ is a constant depending on the geometry.
\end{cor}
We remark that our method also works for other holomorphic bundles such as the canonical line bundle $K_M$ and its multiples. 

\medskip
Next we study the existence of holomorphic functions and sections with polynomial growth on gradient K\"ahler Ricci shrinkers. In this article we prove the following existence result for sections of $K_M^{-q}$ on asymptotically conical K\"ahler Ricci shrinkers.  
\begin{thm}\label{thm: existence of polynomial growth holomorphic sections - introduction}
Let $(M, g, f)$ be a complete gradient K\"ahler Ricci shrinker which is asymptotically conical. There exists a constant $\bar{q}$, such that for any given $q\geq \bar{q}$, there exists a nontrivial holomorphic section $u$ in $K_M^{-q}$ with polynomial growth of order at most $\Lambda$, where $\Lambda$ depends only on $n, q$ and the geometry. Moreover, $\mathcal{O}_\Lambda(K_M^{-q})$ separates points and tangents.
\end{thm}
This theorem is similar in spirit to the existence results in \cite{Liu2016} , \cite{Xu2016} and \cite{Huang2019}. We use the $L^2$ method to obtain local holomorphic sections with certain properties. Then, instead of using scaling arguments as in \cite{Liu2016}, \cite{Xu2016} and \cite{Huang2019} where one has the scaling invariant assumption of nonnegative Ricci curvature, we pullback local sections to larger regions by the biholomorphism generated by $-\n f$.

The frequency function also plays an important role in our method.
Colding and Minicozzi introduced a frequency function for harmonic functions and sections on complete Riemannian manifolds in {\cite{CM1997}} and {\cite{CM1998}}. They generalized it to eigensections of $f$-Laplacian on Ricci shrinkers in {\cite{CM2021}} which is a main tool in the proof of Theorem {\ref{thm: estimate of spectral counting number}}.
The monotonicity of the frequency is also used to ensure that these pullback sections converge to a global section with polynomial growth in the proof of Theorem {\ref{thm: existence of polynomial growth holomorphic sections - introduction}}. In our proof, the asymptotically conical assumption is needed to prove the monotonicity of the frequency for holomorphic sections, see Theorem \ref{eqn: weak monotonicity of U in the AC case}; and the quadratic decay of curvature is used to control the evolution of quantities along the flow lines of $\n f$.

As an application of Corollary \ref{cor: dimension estimate for power of anticanonical bundle} and Theorem \ref{thm: existence of polynomial growth holomorphic sections - introduction}, we obtain the following result.  
 
\begin{cor}\label{cor: Kodaira emdedding of AC shrinkers}
Let $(M, g, f)$ be a gradient K\"ahler Ricci shrinker which is asymptotically conical. Then for $q$ large enough, there exist constants $\Lambda$ and $N$, such that the Kodaira map $\mathcal{K}: M \to \mathbb{CP}^N$ constructed by sections of $\mathcal{O}_{\Lambda}(K_M^{-q})$ is a holomorphic embedding. 
\end{cor}

\begin{remark}
    It has been proved by Conlon, Deruelle and Sun \cite{CDS2024}, among other important results, that asymptotically conical gradient K\"ahler Ricci shrinkers can be holomorphically embedded onto a quasi-projective variety in $\mathbb{CP}^N$ for $N$ large enough. The proof in \cite{CDS2024} involves some algebraic geometric arguments. This result has been improved very recently by Sun and Zhang \cite{SZ2024}, they proved that all complete gradient K\"ahler Ricci shrinkers are quasi-projective, also using algebraic geometry. Quasi-projectivity of K\"ahler manifolds with positive Ricci curvature and curvature decay conditions have been studied by Mok \cite{Mok1990} and Liu \cite{Liu2021b}. Theorem \ref{thm: existence of polynomial growth holomorphic sections - introduction} and Corollary \ref{cor: Kodaira emdedding of AC shrinkers} is proved by a classical analytic method which mainly deals with the embeddedness of the Kodaira map. 
\end{remark}

\begin{remark}
    As a final remark, let's clarify the meaning of ``depending on the geometry" in the statements above. In Theorem \ref{thm: estimate of spectral counting number}, Corollary \ref{cor: estimate of the dimension of holomorphic forms} and Corollary \ref{cor: dimension estimate for power of anticanonical bundle}, the constant $C$ actually depends on the dimension, the curvature bound, the volume noncollapsing constant (shrinkers are known to be non-collapsed \cite{CN2009}), and the volume ratio upper bound $sup_{r> 0} r^{-n} Vol(B(p,r))$ for some fixed point $p$ (which is finite due to \cite{CZ2010}). However, in Theorem \ref{thm: existence of polynomial growth holomorphic sections - introduction} and Corollary \ref{cor: Kodaira emdedding of AC shrinkers}, we need the compactness of a bounded region to ensure that the constants $\bar{q},\Lambda$ and $N$ are finite. 
\end{remark}

The structure of this article is as follows: We prove Theorem \ref{thm: estimate of spectral counting number} in Section \ref{section: spectral counting number}. Corollary \ref{cor: estimate of the dimension of holomorphic forms} and \ref{cor: dimension estimate for power of anticanonical bundle} are proved in Section \ref{section: holomorphic sections}. In Section \ref{section: frequency on AC shrinkers} we study the frequency monotonicity for holomorphic sections of the pluri-anticanonical bundles on asymptotically conical K\"ahler Ricci shrinkers. Finally in Section \ref{section: existence} we prove Theorem \ref{thm: existence of polynomial growth holomorphic sections - introduction} and Corollary \ref{cor: Kodaira emdedding of AC shrinkers}. 


\section{The spectral counting number}\label{section: spectral counting number}
\subsection{Definition and an upper estimate for the frequency.}
Let $(M^n, g, f)$ be a complete gradient Ricci shrinker of real dimension $n$. Throughout the article, we let $b = 2\sqrt{f}$, let $S$ be the scalar curvature, and we always suppose the potential function $f$ is normalized so that 
\begin{equation}\label{eqn: soliton equation for gradient f}
|\n f|^2 + S = f.
\end{equation}
It was proved by Cao and Zhou \cite{CZ2010} that $f$ grows like a quadratic function of the distance
\[
\frac{1}{4}(r(x) - c_1)^2_+ \leq f(x) \leq \frac{1}{4} (r(x) + c_2)^2,
\]
where $c_1, c_2$ are constants depending only on the dimension, $r(x)$ is the geodesic distance to a fixed point where $f$ achieves its minimal value. 
By a well-known result due to Chen \cite{Chen2009}, the scalar curvature is positive everywhere unless the soliton is flat, hence we can assume without loss of generality that $S> 0$ everywhere, then we have $b > 0$.
Note that 
\[
|\n b|^2  = 1 - \frac{S}{4b^2},
\]
and we have assumed that the curvature is uniformly bounded, so $|\n b| > 0$ when $b > R_0$ for some $R_0$ depending on the curvature bound. 

Let $E \to M$ be a vector bundle over $M$ equipped with a metric $\langle, \rangle$ and a compatible connection $\nabla$. Consider the operator $L$ acting on a section $u \in \Gamma(E)$ by 
\begin{equation}\label{eqn: type of elliptic operator}
L u = - \Delta u  + \n_{\n f}u + A(u),
\end{equation}
where $\Delta$ is the rough Laplacian, $A\in \Gamma(End(E))$ is a section of the endomorphism bundle of $E$. We require that $A$ is symmetric 
\[
\la A(u), v\ra = \la u, A(u)\ra, \quad \forall u, v \in \Gamma(E),
\]
and in addition we require that $A$ is bounded, i.e. 
$$|\la A(u), u\ra| \leq \Lambda |u|^2$$ 
for some constant $\Lambda$, where $|u| = \sqrt{\langle u, u\rangle} $. 

Then $L$ is a symmetric operator which can be densely defined on $L^2(dv_f, \Gamma(E))$, the Hilbert space of $L^2$ sections w.r.t. the weighted measure $e^{-f}dv$. By the same argument as in \cite{HN2014} and \cite{CZ2016}, the operator $L$ has discrete spectrum. Moreover, the spectrum of $L$ is bounded from below since we have assumed that $A$ is uniformly bounded. Let $\lambda_1$ be the smallest eigenvalue of $L$, without loss of generality, we can take $\Lambda$ larger if necessary so that
\[
\lambda_1 + \Lambda \geq \frac{1}{2},
\]
which will be assumed throught the following discussion. 

Now let $\lambda$ be an eigenvalue of $L$, and let $u$ be a $W^{1,2}(dv_f)$ eigensection such that $L u = \lambda u$. The norm of $u$ satisfies
\begin{equation}\label{eqn: differential inequality for the norm square}
\Delta_f |u|^2 = 2 \la \Delta u - \n_{\n f} u, u\ra + 2 |\n u|^2 = 2 |\n u|^2 + 2\la A(u) - \lambda u, u\ra \geq -2(\lambda + \Lambda) |u|^2.
\end{equation}

The frequency function for a section $u$ satisfying (\ref{eqn: differential inequality for the norm square}) was studied in \cite{CM2021}. Following \cite{CM2021} we define
\[
I(r) = r^{1-n}\int_{b = r} |u|^2 |\n b|
\]
for regular values $r$ of $b$. The positivity of $I(r)$ has been proved in \cite{CM2021} in more generality. By the divergence theorem
\[
\begin{split}
I(r)= & \int_{b< r} div ( |u|^2 b^{1-n} \n b) = \int_{b< r} {\color{blue}b^{1-n}}\langle \n |u|^2 , \n b \rangle + |u|^2 ((1-n)b^{-n} |\n b|^2 + b^{1-n} \Delta b ).
\end{split}
\]
By equations (\ref{eqn: soliton equation}) and (\ref{eqn: soliton equation for gradient f}), we have the identity
\[
b \Delta b = n - |\n b|^2 - 2S,
\]
thus at a regular value $r$ we have 
\begin{equation}\label{eqn: derivative of I - eigensection}
I'(r) = r^{1-n} \int_{b = r} \langle \n |u|^2 , \vn \rangle + r^{-n} \left( \frac{4n}{r^2} - 2\right) \int_{b = r} \frac{S |u|^2}{|\n b|},
\end{equation}
where $\vn$ is the unit outer normal vector. 
Define
\[
D(r) = \frac{1}{2} r^{2-n} \int_{b = r} \langle \n |u|^2 , \vn \rangle.
\]
By {\color{blue} the} divergence theorem
\begin{equation}\label{eqn: formula of D}
\begin{split}
D(r) =& \frac{1}{2} r^{2-n} e^{r^2/4} \int_{b < r} \Delta_f |u|^2 dv_f \\
= &  r^{2-n} e^{r^2/4} \int_{b < r}  |\n u|^2 + \la A (u) - L u, u\ra  dv_f .\\
\end{split}
\end{equation}
Define the frequency function for $u$ as 
\[
U(r) = \frac{D(r)}{I(r)}
\]
when $I(r)> 0$.

For an eigensection of the operator $L$, we can estimate its frequency by the corresponding eigenvalue. The following theorem can be seen as a weaker version of Theorem 0.7 in \cite{CM2021} in the sense that it is less sharp, however we modify the proof in \cite{CM2021} to make explicit the dependence of the constant $R_0$ on $\lambda + \Lambda$, which is needed for our application. 
 
\begin{thm}\label{thm: frequency upper bound}[Colding-Minicozzi \cite{CM2021}] Let $u$ be a section satisfying (\ref{eqn: differential inequality for the norm square}), where $\lambda + \Lambda \geq 1/2$.
For any $\epsilon > 0$, there is a constant $C_0(n, \epsilon)$, such that for all $r > R_0 := C_0(n, \epsilon )(1 + \sqrt{\lambda + \Lambda})$, we have
\[
U(r) \leq (2+\epsilon) (\lambda + \Lambda).
\]
\end{thm}
\begin{proof} We present a proof following roughly the arguments in \cite{CM2021}, the main difference is in the claim below where we make explicit $R_0$. Take derivative of (\ref{eqn: formula of D}), and use the assumption that $\la A(u), u \ra \leq \Lambda |u|^2$, we have
\begin{equation}\label{eqn: derivative of D}
D'(r) \geq - \frac{n-2}{r} D + \frac{r}{2} D - (\lambda + \Lambda ) r^{2-n} \int_{b = r} |u|^2 |\n b|^{-1} + r^{2-n} \int_{b = r} |\n u|^2 |\n b|^{-1}.
\end{equation}
By Holder inequality, 
\[
D(r)^2  = \left( \frac{1}{2} r^{2-n} \int_{b = r} \langle \n |u |^2 , \vn \rangle  \right)^2 \leq r^{3-n} \left( \int_{b = r} |\n u|^2 |\n b|^{-1}\right) I(r). 
\]
Hence
\begin{equation}\label{eqn: differential inequality for D}
D'(r) \geq - \frac{n-2}{r} D + \frac{r}{2} D  + \frac{U D}{r }- (\lambda + \Lambda ) r I(r) - 4(\lambda + \Lambda ) r^{-n} \int_{b = r} S|u|^2 |\n b|^{-1} .
\end{equation}
Then we have
\begin{equation}\label{eqn: differential inequality for U}
(\log U)'(r) \geq - \frac{n-2}{r} + \frac{r}{2} - \frac{U}{r} - \frac{(\lambda + \Lambda ) r}{U} + 2 \left( 1 - \frac{2(\lambda + \Lambda )}{U} - \frac{2n}{r^2} \right) \frac{1}{r^n I} \int_{b = r} S|u|^2 |\n b|^{-1}. 
\end{equation}
\textbf{Claim:} For any $\epsilon > 0$, there is a constant $C_0(n, \epsilon)$, such that if $U(r_0) > (2+\epsilon)(\lambda + \Lambda)$ for some $r_0 > R_0 := C_0(n, \epsilon )(1 + \sqrt{\lambda + \Lambda})$, then $U(r) > \frac{1}{2} r^2 - r$ for all $r >C_0(n, \epsilon) r_0$.

\textit{Proof of Claim:} Let $R_0 > \sqrt{2n (1+ \frac{2}{\epsilon})}$, if $U(r) > (2+\epsilon)(\lambda + \Lambda)$ for a regular value $r$, then (\ref{eqn: differential inequality for U}) implies
\[
(\log U)'(r) \geq - \frac{n-2}{r}- \frac{U}{r} + \frac{\epsilon r}{2(2+\epsilon)}.
\]
Let $(r_0, r_1)$ be the maximal interval such that 
\begin{equation}\label{eqn: assumed range of U(r)}
(2+\epsilon)(\lambda + \Lambda) <U(r) < \frac{\epsilon r^2}{4(2+\epsilon)}.
\end{equation}
If such an interval does not exist, we can simply take $r_1 = r_0$. In the case $r_1 > r_0$,  for any $r_0 < r < r_1$, we have
\[
(\log U)'(r) \geq - \frac{n-2}{r} + \frac{\epsilon r}{4(2+\epsilon)}\geq \frac{\epsilon r}{ 8 (2+ \epsilon)},
\]
where the last inequality comes from requring that $R_0 > \sqrt{8(2+\epsilon)(n-2)\epsilon^{-1}}$.
Integrating this inequality yields that
\[
U(r) \geq U(r_0) e^{\epsilon 16^{-1} (2+\epsilon)^{-1} (r^2 - r_0^2)} ,
\]
together with (\ref{eqn: assumed range of U(r)}), this implies 
\[
\frac{\epsilon r^2}{4(2+\epsilon)} \geq \frac{2+\epsilon}{2}e^{\epsilon 16^{-1} (2+\epsilon)^{-1} (r^2 - r_0^2)}, 
\]
hence we must have $r_1 \leq C(\epsilon) r_0$ for some constant $C(\epsilon)$ depending on $\epsilon$. Note that we used the assumption that $\lambda + \Lambda \geq 1/2$, otherwise the constant would also depend on the lower bound of the spectrum. 

If there is an $r_2 > r_1$, such that $U(r) > \frac{\epsilon r^2}{5(2+\epsilon)}$ for all $r \in (r_1, r_2)$,  and $U(r_2)\leq \frac{\epsilon r_2^2}{5(2+\epsilon)}$, then we have at $r = r_2$,
\[
\frac{2\epsilon r_2}{5(2+\epsilon)} \geq U'(r_2) \geq  U(r_2) \left( -\frac{n-2}{r_2} - \frac{U(r_2)}{r_2} + \frac{\epsilon r_2}{2 (2+\epsilon)}\right) \geq \frac{\epsilon^2 r_2^3}{40(2+\epsilon)^2},
\]
where we further require that $R_0 > \sqrt{10n (1+ \frac{2}{\epsilon})}$ to have the last inequality. This leads to a contradiction when we take $R_0$ larger if necessary such that $R_0 > 4\sqrt{(2+\epsilon) \epsilon^{-1}}$. Therefore such an $r_2$ cannot exist, and we have $U(r) > \frac{\epsilon r^2}{5(2+\epsilon)}$ for all $r > r_1$.

 If $U(r) > \frac{r^2}{2} - r$ for all $r > r_1$, the proof is done since we have $r_1 < C(\epsilon) r_0$, otherwise there is a largest interval $(r_1, r_3)$, where $\frac{r^2}{2} - r \geq U(r)> \frac{\epsilon r^2}{5(2+\epsilon )}$.  Then on this interval we have
\[
\begin{split}
(\log U)'(r) \geq & - \frac{n-2}{r} + \frac{r}{2} - \frac{U}{r} - \frac{5(2+\epsilon)(\lambda + \Lambda ) }{\epsilon r}\\
\geq & 1 -\frac{n-2}{r} - \frac{5(2+\epsilon)(\lambda + \Lambda ) }{\epsilon r} > \frac{1}{2},
\end{split}
\]
where we further require $C_0(n, \epsilon)$ to be large enough. Integrating the above inequality yields
\[
U(r) \geq U(r_1) e^{\frac{r-r_1}{2}},
\]
which implies that $r_3 < C r_1$ for some universal constant $C$. Then we can use the same argument as above to show that $U(r) > \frac{r^2}{2} - r$ for all $r > r_3$. This finishes the proof of the Claim.

Now, we can prove the theorem by an argument of contradiction. Assume there exists an $r_0 > R_0$, such that $U(r_0) > (2+\epsilon) (\lambda + \Lambda)$, then by the Claim, for $r > (4+C_0(n, \epsilon) )r_0$, we have $K(r): = D(r) - 4(\lambda + \Lambda )I(r) > 0$. By (\ref{eqn: derivative of I - eigensection}) and (\ref{eqn: differential inequality for D}), we have 
\[
r K' \geq (2 - n + r^2 - r -8(\lambda +\Lambda))K + 4(\lambda + \Lambda) (2 - n + \frac{3r^2}{4} - r - 8\lambda) I,
\]
see \cite{CM2021} for details of the calculation. Observe that when $r >C(n) + 4\sqrt{\lambda + \Lambda}$ for some dimensional constant $C(n)$, the above inequality implies
\[
r K' \geq \frac{3r^2}{4}K.
\]
Integrating the above inequality yields
\[
K(r) > K(s) e^{\frac{3}{8} (r^2 - s^2)}, \quad r > s > C(n) + 4\sqrt{\lambda + \Lambda}+ (C_0+4)r_0.
\]
Now
\[
D(r) = \frac{1}{2} r^{2-n} e^{r^2/4} \int_{b < r } \Delta_f|u|^2 dv_f\geq  K(s) e^{\frac{3}{8} (r^2 - s^2)},
\]
hence
\[
\lim_{r \to \infty} \int_{b < r } \Delta_f|u|^2 dv_f \geq \lim_{r \to \infty} c e^{r^2 / 8} = \infty,
\]
however, by (\ref{eqn: formula of D}), the left hand side should be finite since $u$ is in $W^{1,2}(dv_f)$. 
\end{proof}

\subsection{Estimate of the spectral counting number}
We prove Theorem \ref{thm: estimate of spectral counting number} in this subsection. First we need the following mean-value inequality.
\begin{lem}\label{lem: mean value inequality} Let $(M^n, g, f)$ be a complete gradient Ricci shrinker with bounded curvature, suppose $u$ is a nonnegative solution of the differential inequality
\[
\Delta_f u \geq - a u,
\]
where $a \geq 0$ is a nonnegative constant. Then for any $x \in M$ and $0 < r < 1$, we have
\[
|u(x)|^2 \leq \frac{C_M(a+1+d(p,x))}{Vol(B_x(r))} \int_{B_x(r)} |u|^2,
\]
where $p$ is a fixed minimal point of $f$, $C_M$ is a constant depending on the dimension $n$ and the curvature bound.
\end{lem}
\begin{proof}
Since we have assumed bounded curvature,  by results in {\cite{LS1984}} and {\cite{Sal1992}}, for any $r < 1$, there is a Sobolev inequality
\[
\left( \int_{B_x(r)} u^{\frac{2n}{n-2}} \right)^{\frac{n-2}{n}} \leq \frac{C_S r^2}{Vol(B_x(r))^{2/n}} \int_{B_x(r)} |\n u|^2,
\]
for any $u\in C_0^1(B_x(r))$, where $p$ is any point on $M$, where the Sobolev constant $C_S$ depends on the dimension and the curvature bound. This Sobolev inequality leads to a mean value inequality for nonnegative solutions of the differential inequality
\begin{equation}\label{eqn: a differential inequality for f-Laplacian}
\Delta_f u \geq - a u.
\end{equation}
The proof is by Moser-iteration, the only difference to the standard case is the drift term $\la \n f, \n u\ra$, which can be handled as follows: 

Let $\phi$ be any cut-off function supported on $B_x(r)$, then for any nonnegative $u$ satisfying (\ref{eqn: a differential inequality for f-Laplacian}), and any $p \geq 1$, we have 
\[
\begin{split}
a \int u^{2p} \phi^2 \geq & \int - \phi^2 u^{2p-1}\Delta u + \phi^2 u^{2p-1}\langle \n f, \n u\rangle \\
= & \int (2p-1)\phi^2 u^{2p-2} |\n u|^2 + 2 \phi u^{2p-1} \langle \n \phi, \n u\rangle + \phi^2 u^{2p-1}\langle \n f, \n u\rangle. \\
\end{split}
\]
By Cauchy-Schwarz inequality, we get
\[
(2p-\frac{3}{2})\int \phi^2 u^{2p-2} |\n u|^2 \leq a \int \phi^2 u^{2p} + 4 \int |\n \phi|^2 u^{2p} + \int_{spt \phi} |\n f|^2 u^{2p}.
\]
On a gradient Ricci shrinker we have $|\n f| \leq \frac{1}{2} d(p, \cdot) + c(n)$, where $p$ is a  minimal point of $f$. Then the standard Moser iteration yields the result. 
\end{proof}

Now we can prove the estimate of the spectral counting number. 

{
\begin{proof}[Proof of Theorem \ref{thm: estimate of spectral counting number} ]
Take $R_0 = C(n, \epsilon)(1+\sqrt{\lambda +\Lambda})$ as in Theorem \ref{thm: frequency upper bound}, let $\alpha = 1+ \frac{2}{\lambda + \Lambda}$. We can take $C(n, \epsilon)$ larger if necessary, so that 
\[
1-\epsilon < |\n b|
\]
when $b\geq R_0$.
Let $u$ be any finite linear combination of eigensections with eigenvalue at most $\lambda$, define the quantities $I(r), D(r), U(r)$ as in the previous section. Then by (\ref{eqn: derivative of I - eigensection}), we have 
\[
 I'(r) \leq \frac{2D(r)}{r}
\]
for all $r> R_0$. Hence by Theorem \ref{thm: frequency upper bound}, we have
\[
 (\ln I)' \leq \frac{2U}{r} \leq  \frac{2(2+\epsilon)(\lambda + \Lambda)}{r}.
\]
Integrating the above inequality from $r$ to $\alpha r$ yields
\begin{equation}\label{eqn: doubling control of I}
\frac{I(\alpha r)}{I(r)} \leq \alpha^{2(2+\epsilon) (\lambda +\Lambda)}.
\end{equation}
Let $r_0 = R_0$, $r_i = \alpha r_{i-1}$ for $i = 1,2.$ 
Define
\[
J_i = \int_{r_0}^{r_i} s^{n-1} I(s) ds = \int_{r_0< b < r_i} |u|^2 |\n b|^2 .
\]
By a change of variable, we have
\[
J_{i+1} - J_i = \int_{r_i}^{r_{i+1}} s^{n-1}I(s) ds = \int_{r_{i-1}}^{r_i} \alpha^{n-1} t^{n-1}I(\alpha t) \alpha dt ,
\]
hence by (\ref{eqn: doubling control of I}) we have
\[
{J_{2} - J_1 }\leq  \alpha^{n+ 2(2+\epsilon) (\lambda +\Lambda)}{ (J_1 - J_{0})} 
\]
where $J_0 = 0$, from which we can derive that
\begin{equation}\label{eqn: doubling on balls}
\int_{b < r_2} |u|^2 |\n b|^2 = \int_{b< r_0} |u|^2 |\n b|^2 + J_1 + (J_2 - J_1) \leq (n+\alpha^{n+ 2(2+\epsilon) (\lambda +\Lambda)}) \int_{b < r_1} |u|^2 |\n b|^2.
\end{equation}
Define an inner product on the space of sections that are $L^2$ on $\{b < r_2\}$ by
\[
K(u, v) = \int_{b< r_2} \langle u, v \rangle |\n b|^2. 
\]
Let $\mathcal{E}_\lambda$ be the linear space generated by eigensections with eigenvalues at most $\lambda$,  then $\dim \mathcal{E}_\lambda = \mathcal{N}(\lambda)$.
Let $u_1, u_2, ..., u_{\mathcal{N}(\lambda)}$ be a basis of $\mathcal{E}_\lambda$, which are orthonormal w.r.t. the inner product $K(,)$. Define
\[
F(x) : = \sum_{i = 1}^{ \mathcal{N}(\lambda) } |u_i|^2(x).
\]
Note that $F(x)$ is independent of the choice of orthonormal basis $u_1, u_2, ..., u_{\mathcal{N}(\lambda)}$. We have 
\[
\mathcal{N}(\lambda) = \int_{b< r_2} F \leq (1+\alpha^{n+ 2(2+\epsilon)(\lambda + \Lambda)} )\int_{b< r_1} F,
\]
where we have applied (\ref{eqn: doubling on balls}) to each of $u_i$, $i = 1,2,..., \mathcal{N}(\lambda)$.

Let $x$ be the maximum point of $F$ in $\{b< r_1\}$, there is a subspace $\mathcal{E}_\lambda(x)$ of codimension at most $rank(E)$ consisting of sections vanishing at $x$ (see \cite{Li1997}). Up to an orthonormal change of basis, we can assume WLOG that $u_1, u_2, ..., u_k \in \mathcal{E}_\lambda(x)$, and the rest of the basis vectors are orthogonal to $\mathcal{E}_\lambda(x)$, hence $\mathcal{N}(\lambda) - k \leq rank E$. Then we have
\[
\int_{b< r_1} F \leq Vol(b< r_1)\sum_{i = 1}^{ \mathcal{N}(\lambda) } |u_i|^2(x) =  Vol(b< r_1) \sum_{i = k+1}^{ \mathcal{N}(\lambda) } |u_i|^2(x).
\]
Note that by (\ref{eqn: differential inequality for the norm square}), Lemma \ref{lem: mean value inequality} can be applied to each $u_i, i = k+1,..., \mathcal{N}(\lambda)$ on $B_x(\frac{1}{\sqrt{\lambda + \Lambda}})$ to get
\[
|u_i|^2(x) \leq \frac{C_M(1+\lambda + \Lambda + d(p, x))}{Vol(B_x(\frac{1}{\sqrt{\lambda + \Lambda} }))} \int_{B_x(\frac{1}{\sqrt{\lambda + \Lambda}})} |u_i|^2.
\]
Note that for any $y \in B_x(\frac{1}{\sqrt{\lambda + \Lambda}})$, we have $b(y) \leq b(x) + \frac{1}{\sqrt{\lambda +
\Lambda}}$, and $r_2 - r_1 \geq (\alpha - 1) R_0 \geq 2C(n, \epsilon) \frac{1}{\sqrt{\lambda + \Lambda}}$, where $C(n, \epsilon)$ is taken to be a large constant, hence $B_x(\frac{1}{\sqrt{\lambda+\Lambda}}) \subset \{b< r_2\}$. Then we have 
\[
\int_{B_x(\frac{1}{\sqrt{\lambda + \Lambda}})} |u_i|^2 \leq 1+\epsilon,
\] since $u_i$ is normalized on $\{b< r_2\}$ with respect to the inner product $\mathcal{K}(,)$.

By the assumption of bounded curvature, and since gradient shrinking Ricci solitons are noncollapased \cite{CN2009}, there exists a constant $v_0 > 0$ such that 
\[
Vol(B_x(\frac{1}{\sqrt{\lambda + \Lambda} })) \geq \frac{v_0}{(\sqrt{\lambda +\Lambda})^n}.
\]
Therefore, the estimates above yield that 

\[
\mathcal{N}(\lambda) \leq rank(E) C_1(1 + \lambda + \Lambda + \alpha^2 R_0) (\lambda + \Lambda)^{\frac{n}{2}} (1 + \alpha^{n+ 2(2+\epsilon)(\lambda + \Lambda)}) Vol(b < r_1),
\]
where $C_1$ is a constant depending on $n, v_0$ and the curvature bound. 
Now, fix a small value for $\epsilon$ and note that $ R_0 = C_0(n, \epsilon )(1 + \sqrt{\lambda + \Lambda})$, the constant $\alpha^{\lambda +\Lambda}$ is bounded, and by \cite{CZ2010} we have
\[
Vol(b< r_1) \leq C  r_1^n \leq C \alpha^{n}R_0^n,
\]
thus we get
\[
\mathcal{N}(\lambda) \leq C rank(E) (1+\lambda + \Lambda)^{n+1}
\]
for some constant $C$ depending on the geometry. 
\end{proof}

}
\section{Holomorphic sections}\label{section: holomorphic sections}
\subsection{General discussion}
In this section, let $(M^{2m}, g, f)$ be a gradient K\"ahler Ricci shrinker with complex dimension $m$ and real dimension $n = 2m$. In order to use the result{\color{blue}s} in the previous sections, we want to relate holomorphic sections with polynomial growth to eigensections of some Elliptic operator on the gradient Ricci shrinker. We have already done this for holomorphic functions and $(p,0)$-forms in our previous work \cite{HO2024}. Similar idea has been used in \cite{CM2020} to relate eigen-functions of a drifted Laplacian on a mean curvature self-shrinker to ancient caloric functions along the associated mean curvature flow. For sections of a Hermitian vector bundle it is more complicated, and the choice of the Hermitian metric on the bundle is important. 

Let's first do some general discussion before turning to specific cases. Let $E \to M$ be a holomorphic vector bundle over $M$ equipped with a Hermitian metric $\la, \ra$, let $\n$ be the unique Hermitian connection on $E$ such that $\n$ is compatible with the metric $\la ,\ra$, and satisfies $\n^{0,1} = \dbar$. 

Let $u$ be a holomorphic section of $E$, locally we can write $u = u_i  e_i$, where $\{e_i\}$ is a local holomorphic frame of $E$. Then the functions $u_i$ are holomorphic, hence we have 
\begin{equation}\label{eqn: laplacian of a holomorphic section}
\Delta u = \Delta u_i  e_i - tr \Theta (u) =  -tr \Theta (u),
\end{equation}
where $\Theta$ is the endomorphism-valued curvature $(1,1)$-form, and we denote
\[
\Delta = g^{i \jb} (\n_i \n_{\jb} + \n_{\jb} \n_i), \quad tr \Theta = g^{i \jb}\sum \Theta ( \pd_i, \pd_{\bar{j}} ),
\]
to be consistant with the notations in the Riemannian case. 

There is a canonical way to associate a Ricci flow solution to a gradient Ricci soliton. Let $\tau(t) = -t$, and let $\Phi_t$ be the family of diffeomorphisms generated by 
\begin{equation}\label{eqn: definition of canonical diffeomorphisms}
\pd_t \Phi_t(x) = \frac{1}{\tau(t)} \n f( \Phi_t(x)), \quad t \in (-\infty, 0), \quad \Phi_{-1} = id.
\end{equation}
Let 
\begin{equation}\label{eqn: canonical Ricci flow}
g(t) = \tau(t) \Phi_t^* g. 
\end{equation}
Then $g(t)$ is a solution of the Ricci flow. Since holomorphicity is independent of $t$, the equation $\Delta_{g(t)} u = -tr_{g(t)} \Theta (u)$ holds for all $t\in (-\infty, 0)$. 
Define
\begin{equation}\label{eqn: definition of u hat}
\hat{u}(x, s) = \left( \Phi_t^{-1} \right)^*u(x), \quad \text{where } t = - e^{-s}, \quad s \in (-\infty, \infty). 
\end{equation}
Note that $\pd_s = \tau(t) \pd_t$. We can calculate that
\begin{equation}\label{eqn: heat type equation for a transformation of holomorphic section}
\begin{split}
\pd_s \hat{u}(x, s) = & \tau \left( (\Delta_{g(t)} u)(\Phi_t^{-1}(x) ) + tr_{g(t)} \Theta(u)(\Phi_t^{-1}(x) )   - \mathcal{L}_{\frac{1}{\tau} \n f} u(\Phi_t^{-1}(x) ) \right) \\
= & \Delta_g \hat{u}(x, s) + tr_{g} \tilde{\Theta}_s (\hat{u}(x, s)) - \mathcal{L}_{\n f} \hat{u} (x, s),
\end{split}
\end{equation}
 where $\tilde{\Theta}_s : = \Theta |_{\Phi^{-1}_{- e^{-s}}}$.
 
We would like to write the operator on the right hand side of (\ref{eqn: heat type equation for a transformation of holomorphic section}) in the form of (\ref{eqn: type of elliptic operator}) in order to use results in the previous sections, however, the term $\tilde{\Theta}_s$ may depend on $s$ in general. 
In some cases when the hermitian metric on $E$ were induced by the K\"ahler metric $g$, we could have $\tilde{\Theta}_s$ independent of $s$. In the rest of this section we consider two special cases: the first one is the bundle $\Lambda^{p,0}$ of $(p,0)$-forms, the second one is $K_M^{-q}$, the $q$-th power of the anti-canonical line bundle, both are equipped with the induced metric. 

\subsection{Holomorphic $(p, 0)$-forms.}

Before turning to holomorphic forms, let's first compute the weighted Hodge Laplacian $\Delta^d_f$ on a smooth Riemannian manifold with a weighted measure $dv_f = e^{-f} dv$. Our purpose is to write it in the form of (\ref{eqn: type of elliptic operator}). In this calculation we do not need the soliton structure. 

Let $d$ be the exterior derivative on differential forms, and let $d_f^*$ be its $L^2(dv_f)$-dual operator, then 
\[
d_f^* = d^* + i_{\n f},
\]
where $d^*$ is the usual $L^2(dg)$-dual of $d$, and $i$ denotes the interior product. Define the weighted Hodge Laplacian to be 
\[
\Delta_f^d = d^*_f d + d d^*_f .
\]
One can check that 
\[
\Delta_f^d = \Delta^d + \mathcal{L}_{\n f},
\]
where $\Delta^d$ is the usual Hodge Laplacian and $\mathcal{L}$ is the Lie derivative. 

Let $\Delta = tr \n^2$ be the rough Laplacian, we have the well-known Weitzenbock formula
\[
\Delta^d  \w = - \Delta \w+ \mathcal{R}(\w),
\]
where $\mathcal{R}(\w)$ is the curvature term.

Let $\langle \cdot, \cdot \rangle$ be the inner product on $\Lambda^p$ induced by the Riemannian metric, and choose an orthonormal coframe $\theta^1, \theta^2, ..., \theta^n$ w.r.t. this inner product. Let $\w$ be a $p$-form written as
\[
\w = \frac{1}{p !} \sum_{i_1 i_2,..., i_p}\w_{i_1 i_2,..., i_p} \theta^{i_1} \wedge \theta^{i_2} \wedge ... \wedge \theta^{i_p},
\]
and let $\eta$ be another $p$-form written similarly, then we have the local formula
\[
\langle \w, \eta \rangle =  \frac{1}{p !} \sum_{i_1 i_2,..., i_p} \w_{i_1 i_2,..., i_p} \eta_{i_1 i_2,..., i_p},
\]
and the curvature term can be written as
\[
\mathcal{R}(\w) = \sum R_{i_k j_k i_l j_l} \w_{i_1 ...j_k...i_p} \theta^{i_1} \wedge ... \wedge \theta^{j_l} \wedge ...\wedge \theta^{i_{p+1}}.
\]
By Cartan's magic formula that $\mathcal{L}_X = i_X d + d i_X$, we have
\begin{equation}\label{eqn: eigen-equation}
\Delta_f^d \w = - \Delta \w  + \mathcal{R}(\w) + i_{\n f} d \w + d i_{\n f} \w.
\end{equation}
Define
\begin{equation}\label{eqn: definition of R_f}
\mathcal{R}_f(\w) = \mathcal{R}(\w) + \frac{1}{p!}\sum_{}f_{i_1 i_2} \w_{i_2 i_3 ... i_{p+1} }  \theta^{i_1} \wedge \theta^{i_3}\wedge .. \wedge \theta^{i_{p+1}}.
\end{equation}
\begin{lem}\label{lem: rewriting the f-Hodge Laplacian}(Weighted Weitzenbock formula.)
\[
\Delta_f^d \w = - \Delta \w + \n_{\n f} \w + \mathcal{R}_f(\w). 
\]
\end{lem}
\begin{proof}
The proof is by direct calculation. First
\[
i_{\n f} \w = \frac{1}{(p-1)!} \sum_{i_1, i_2, ..., i_p } f_{i_1} \w_{i_1 i_2 ... i_p} \theta^{i_2} \wedge ... \wedge \theta^{i_p}.
\]
Taking exterior derivative yields
\[
d i_{\n f} \w = \frac{1}{(p-1)!} \sum_{i_1, i_2, ..., i_p } \sum_k \left(  f_{i_1 k} \w_{k  i_2 ... i_p} + f_k \n_{i_1}\w_{k i_2 ... i_p} \right) \theta^{i_1} \wedge \theta^{i_2} \wedge ... \wedge \theta^{i_p}.
\]
Observe that
\[
\begin{split}
d\w = & \sum_{k} \theta^k \wedge \n_k \w\\
= & \frac{1}{p !} \sum_{i_1 i_2,..., i_p, i_{p+1}} \n_{i_1 }\w_{i_2, ..., i_{p+1}} \theta^{i_1}\wedge \theta^{i_2} \wedge ... \wedge \theta^{i_{p+1}} \\
= & \frac{1}{ (p+1)!} \sum_{i_1 i_2,..., i_p, i_{p+1}} \left( \frac{1}{p !} \sum_{\sigma \in S_{p+1}} sgn(\sigma) \n_{i_{\sigma(1)} }\w_{i_{\sigma(2)}, ..., i_{\sigma (p+1)}} \right) \theta^{i_1}\wedge \theta^{i_2} \wedge ... \wedge \theta^{i_{p+1}}.
\end{split}
\]
Thus
\[
\begin{split}
i_{\n f} d\w = & \frac{1}{p!} \sum_{i_1, i_2, ..., i_{p+1}} f_{i_1} \left( \frac{1}{p!} \sum_{\sigma \in S_{p+1}} sgn(\sigma)\n_{i_{\sigma(1)}} \w_{i_{\sigma(2)} ..., i_{\sigma(p+1)}}  \right) \theta^{i_2} \wedge ... \wedge \theta^{i_{p+1}}, \\
\end{split}
\]
we break the second summation according to the value of $\sigma(1)$ to get
\[
\begin{split}
i_{\n f} d\w= &  \frac{1}{p!} \sum_{i_1, i_2, ..., i_{p+1}} f_{i_1} \left( \frac{1}{p!} \sum_{\sigma \in S_{p+1}, \sigma(1) = 1} sgn(\sigma)\n_{i_{1}} \w_{i_{\sigma(2)} ..., i_{\sigma(p+1)}}  \right) \theta^{i_2} \wedge ... \wedge \theta^{i_{p+1}}\\
 &+  \frac{1}{p!} \sum_{i_1, i_2, ..., i_{p+1}} f_{i_1} \left( \frac{1}{p!} \sum_{\sigma \in S_{p+1}, \sigma(1) = 2} sgn(\sigma)\n_{i_{2}} \w_{ i_{\sigma(2)} ..., i_{\sigma(p+1)}}  \right) \theta^{i_2} \wedge ... \wedge \theta^{i_{p+1}}\\
&+ ...\\
& +  \frac{1}{p!} \sum_{i_1, i_2, ..., i_{p+1}} f_{i_1} \left( \frac{1}{p!} \sum_{\sigma \in S_{p+1}, \sigma(1) = p+1} sgn(\sigma)\n_{i_{p+1}} \w_{i_{\sigma(2)} ..., i_{\sigma(p+1)} }  \right) \theta^{i_2} \wedge ... \wedge \theta^{i_{p+1}},\\
\end{split}
\]
which simplifies to 
\[
\begin{split}
i_{\n f} d\w = & \frac{1}{p!} \sum_{i_1, i_2, ..., i_{p+1}} f_{i_1}  \n_{i_{1}} \w_{i_{2} ..., i_{p+1}}   \theta^{i_2} \wedge ... \wedge \theta^{i_{p+1}}\\
 &+  \frac{1}{p!} \sum_{i_1, i_2, ..., i_{p+1}} (-1)f_{i_1} \n_{i_{2}} \w_{i_{1} i_3... i_{p+1}}  \theta^{i_2} \wedge ... \wedge \theta^{i_{p+1}}\\
&+ ...\\
&+  \frac{1}{p!} \sum_{i_1, i_2, ..., i_{p+1}} (-1)^p f_{i_1} \n_{i_{p+1}} \w_{i_1 i_{2} i_3... i_{p} }  \theta^{i_2} \wedge ... \wedge \theta^{i_{p+1}},\\
\end{split}
\]
then we observe that
\[
\begin{split}
i_{\n f} d\w = & \frac{1}{p!} \sum_{i_1, i_2, ..., i_{p+1}} f_{i_1}  \n_{i_{1}} \w_{i_{2} ..., i_{p+1}}   \theta^{i_2} \wedge ... \wedge \theta^{i_{p+1}}\\
&  - \frac{1}{(p-1)!} \sum_{i_1, i_2, ..., i_{p+1}} f_{i_1} \n_{i_{2}} \w_{i_{1} i_3... i_{p+1}}  \theta^{i_2} \wedge ... \wedge \theta^{i_{p+1}}.
\end{split}
\]
Combining the above formulas, we get
\[
\begin{split}
\mathcal{L}_{\n f} \w = & d i_{\n f} \w + i_{\n f } d\w \\
 = &   \frac{1}{p!} \sum_{i_1, i_2, ..., i_{p+1}} \left(  f_{i_2 i_1} \w_{i_1 i_3 ... i_{p+1}} +  f_{i_1}  \n_{i_{1}} \w_{i_{2} ..., i_{p+1}}   \right)\theta^{i_2} \wedge ... \wedge \theta^{i_{p+1}}.
\end{split}
\]
Observe that 
\[
\mathcal{R}(\w) + \mathcal{L}_{\n f} \w = \mathcal{R}_f(\w) + \n_{\n f} \w,
\]
this finishes the proof. 
\end{proof}

Now let's focus on the bundle $\Lambda^{p, 0}$ of $(p, 0)$-forms on a complete K\"ahler Ricci shrinker $(M, g, f)$. It was proved in \cite{MW2015} that $\Delta^d_f$ preserves the type of $(p, q)$-forms since $\n f$ is the real part of a holomorphic vector field on a gradient K\"ahler Ricci shrinker. 
The following lemma was proved in \cite{HO2024},
\begin{lem}\label{lem: mapping holomorphis forms to solutions of heat equation}
Let $\w$ be a holomorphic $(p, 0)$-form on a complete gradient Ricci shinker $(M, g, f)$. Let $\Phi_t$ be the family of diffeomorphisms defined in (\ref{eqn: definition of canonical diffeomorphisms}), the form
\[
\hat{\w} (x, s) := (\Phi^{-1}_{-e^{-s}})^*\w(x) , \quad (x, s) \in M \times (-\infty, \infty)
\] 
is an ancient solution of the heat type equation
\begin{equation}\label{eqn: heat equation for the f Hodge Laplacian}
\pd_s \hat {\w} + \Delta_f^d \hat{\w} = 0,
\end{equation}
where $\Delta^d_f$ is the weighted Hodge Laplacian on $(M, g, f)$. 
And the map sending $\w$ to $\hat{\w}$ is a linear injection from the space of holomorphic $(p, 0)$ forms to the space of $(p, 0)$-forms satisfying (\ref{eqn: heat equation for the f Hodge Laplacian}) for $s\in (-\infty, \infty)$. 
\end{lem}
The proof depends on the self-similar structure of gradient Ricci shrinker, and the fact that holomorphic $(p, 0)$-forms are harmonic on K\"ahler manifolds. 
We have proved in \cite{HO2024} under the assumption of bounded Ricci curvature that $\Delta^d_f$ acting on $L^2(dv_f)$-forms has discrete spectrum, its eigen-forms provide a complete basis for the $L^2(dv_f)$ space.  Moreover, we have shown in \cite{HO2024} that
\begin{lem}\label{lem: linear expansion for solutions of the f heat equation of forms}
Suppose $(M, g, f)$ is a complete gradient K\"ahler Ricci shrinker with bounded Ricci curvature, let $K = \sup_M|Ric| + \frac{1}{2}$. Then for any holomorphic $(p, 0)$-form $\w$ with polynomial growth of order at most $\mu$, there exist finitely many coefficients $a_i , i = 1,2,...$ such that 
\[
\hat{\w} = \sum_{\lambda_i \leq \mu/2 + p K} a_i e^{-\lambda_i s} \theta_i,
\]
where $\lambda_i$, $\theta_i$, $i = 1,2,...$ are the eigenvalues and $(p,0)$-eigenforms of $\Delta^d_f$. 
\end{lem}
As a consequence, we can prove the dimension estimate for holomorphic $(p, 0)$-forms with polynomial growth by its growth order, see Corollary \ref{cor: estimate of the dimension of holomorphic forms} in the introduction.

\begin{proof}[Proof of Corollary \ref{cor: estimate of the dimension of holomorphic forms}.]
By Lemma \ref{lem: mapping holomorphis forms to solutions of heat equation} and \ref{lem: linear expansion for solutions of the f heat equation of forms}, we have
\[
\dim \mathcal{O}_{\mu}(\Lambda^{(p,0)}) \leq \mathcal{N}(\mu/2 + p (\sup_M |Ric| + \frac{1}{2}) ),
\]
where $\mathcal{N}$ is the spectral counting number of $\Delta^d_f$ acting on $L^2(dv_f)$-sections of $\Lambda^{p, 0}$. By Lemma \ref{lem: rewriting the f-Hodge Laplacian}, the operator $\Delta^d_f$ is in the form of (\ref{eqn: type of elliptic operator}). And by (\ref{eqn: definition of R_f}), the endomorphism term $\mathcal{R}_f$ is bounded since we have assumed bounded curvature. The corollary then follows from Theorem \ref{thm: estimate of spectral counting number}.
\end{proof}

\subsection{Holomorphic sections of $K_M^{-q}$.}
Let $K_M^{-1}$ be the anti-canonical line bundle, and let $U_\alpha \subset M^{2m}$ be trivializing open sets for $K_M^{-1}$.  A section of $K_M^{-1}$ can be written as a set of functions $u = (u_\alpha)$, where each $u_\alpha$ is a complex function on $U_\alpha$. Suppose there are holomorphic local coordinates on each $U_\alpha$, in which the K\"ahler metric can be written as $g_{i \jb}^\alpha$. Then the set of functions $ h : =( - \ln \det g_{i \jb}^\alpha )$ defines a Hermitian metric on $K_X^{-1}$. The norm of a section $u$ can be written on a trivializing open set $U_\alpha$ as 
$$
|u|_h^2 := |u_\alpha|^2e^{\ln \det g_{i \jb}^\alpha }.
$$

Naturally, for each positive integer $q$, $q h := (- q \ln \det g_{i \jb})$ defines a metric on the bundle $K_M^{-q}$. For simplicity, we will still denote the norm on $K_M^{-q}$ as $|\cdot|_h$. The curvature $\sqrt{-1}\Theta = - q \sqrt{-1}\pd \dbar \ln \det g = q Ric$ is $q$-times the Ricci form, and
$tr \Theta =   \frac{q}{2} S$,  
where $S = 2 g^{i \jb} R_{i \jb}$ is the Riemannian scalar curvature. 
 
 Let $g(t)$ be the Ricci flow defined in (\ref{eqn: canonical Ricci flow}), let $h(t)$ be the metric on $K_M^{-1}$ induced by $g(t)$. For any holomorphic section $u$ of $K_M^{-q}$, as discussed above (\ref{eqn: laplacian of a holomorphic section}), $u$ satisfies the equation
 \begin{equation}\label{eqn: Laplacian of a holomorphic section}
 \Delta_{g(t)} u = - tr_{g(t)} \Theta_{h(t)} u = -\frac{q}{2}S_{g(t)} u,
 \end{equation}
 for all $t\in (-\infty, 0)$. Note that $S_{g(t)} ( \Phi_t^{-1}(x)) = \frac{1}{\tau(t)} S_{\Phi_t^*g} ( \Phi_t^{-1}(x)) = \frac{1}{\tau(t)} S_g(x)$. Thus the curvature term in (\ref{eqn: heat type equation for a transformation of holomorphic section}) becomes independent of $s$, and $\hat{u}(x, s) = \left( \Phi_{-e^{-s}}^{-1} \right)^*u (x)$ satisfies
 \[
 \pd_s \hat{u}(x, s) = \Delta_g \hat{u}(x, s) +\frac{q}{2} S_g(x) \hat{u}(x, s) - \mathcal{L}_{\n f} \hat{u} (x, s).
 \]
 Define the operator $L = - \Delta_g - \frac{q}{2} S_g + \mathcal{L}_{\n f}$ acting on sections of $K_M^{-q}$. 
\begin{lem}\label{lem: Lie derivative of a section}
For any section $u$ of $K_M^{-q}$, we have 
\[\mathcal{L}_{\n f} u =\n_{\n f} u - \frac{q}{2} \Delta f u.\] 
\end{lem}
\begin{proof}
Let's first compute $\mathcal{L}_{\n f}$ on a section $v$ of the canonical line bundle $K_M$. Let $v = v_\alpha dz^1 \wedge dz^2 \wedge ... \wedge dz^n$ on a trivializing open set $U_\alpha$, then 
\[
\mathcal{L}_{\n f} v = \n_{\n f} v_\alpha dz^1 \wedge dz^2 \wedge ... \wedge dz^n + v_\alpha \mathcal{L}_{\n f} ( dz^1 \wedge dz^2 \wedge ... \wedge dz^n ).
\]
Write $\n f = g^{i\jb} f_i \pd_{\jb} + g^{i\jb} f_{\jb} \pd_i$. By Cartan's formula, 
\[
\begin{split}
\mathcal{L}_{\n f} ( dz^1 \wedge dz^2 \wedge ... \wedge dz^n ) = & d i_{\n f} ( dz^1 \wedge dz^2 \wedge ... \wedge dz^n ) \\
= & d\left( g^{i \jb} f_{\jb} \sum_{i } (-1)^{i-1}dz^1 \wedge dz^2 \wedge ...\wedge \check{dz^i} \wedge ... \wedge dz^n \right)\\
= &g^{i \jb} f_{\jb i} dz^1 \wedge dz^2 \wedge ... \wedge dz^n \\
= &\frac{1}{2}\Delta f dz^1 \wedge dz^2 \wedge ... \wedge dz^n \\
\end{split}
\]
Hence
\[
\mathcal{L}_{\n f} v = \n_{\n f} v + \frac{1}{2}\Delta f v .
\]
Since $K_M^{-1}$ is the dual bundle of $K_M$, we get 
\[
\mathcal{L}_{\n f} u = \n_{\n f} u - \frac{q}{2} \Delta f u ,
\]
for any section $u$ of $K_M^{-q}$.
\end{proof}
Recall that the soliton identity implies $S + \Delta f  = \frac{n}{2}$, where $n$ is the real dimension, hence  by Lemma \ref{lem: Lie derivative of a section} we have 
\[
L u = - \Delta u + \n_{\n f} u - \frac{nq}{4} u,
\]
for sections $u$ of $K_M^{-q}$. Thus $L$ is in the form of (\ref{eqn: type of elliptic operator}), in particular, the endomorphism term is only a multiplication by a constant. Then by the general discussion (\ref{eqn: heat type equation for a transformation of holomorphic section}), the section $\hat{u}$ defined in (\ref{eqn: definition of u hat}) is a solution of the heat type equation
 \begin{equation}\label{eqn: heat type equation for L}
 \pd_s \hat{u} + L \hat{u} = 0, \quad s\in (-\infty, \infty).
 \end{equation}
Obviously the map $u \mapsto \hat{u} $ is linear and injective. 

For each $d> 0$, let $\mathcal{O}_d(K_M^{-q})$ be the space of holomorphic sections of $K_M^{-q}$ with polynomial growth, i.e. for any $u \in \mathcal{O}_d(K_M^{-q})$, there exists a constant $C$ such that
\[
|u|_h (x)\leq C(1+r(x))^d,
\]
for all $x\in M$.

Now let $u \in \mathcal{O}_d(K_M^{-q})$, we want to derive growth estimate for $\hat{u}$. For simplicity, let's denote $\psi_s = \Phi_{-e^{-s}}^{-1}$. Note that
\[
\frac{d}{ds} \left( r(\psi_s(x) ) + c(n) \right) \leq |\n f|(\psi_s(x)) \leq  \frac{1}{2} \left( r(\psi_s(x) ) + c(n) \right),
\]
after integration we get
\[
 r(\psi_s(x) ) + c(n) \leq \left( r(x) + c(n) \right) e^{s/2},  \quad \text{when } s \geq  0;
\]
similarly we have 
\[
 r(\psi_s(x) ) + c(n) \leq \left( r(x) + c(n) \right) e^{- s/2},  \quad \text{when } s <  0.
\]
Similarly we can control the change of volume elements along an integral curve of $-\frac{1}{\tau} \n f$. 
\[
\frac{d}{ds} \det \psi_s^* g  = tr_g \left( \mathcal{L}_{- \n f} g \right) |_{\psi_s} \det \psi_s^* g  = -\Delta f|_{\psi_s} \det \psi_s^* g.
\]
Since $(M, g)$ has nonnegative scalar curvature $S \geq 0$, we have $\frac{n}{2} - \sup_M S \leq \Delta f = \frac{n}{2} -S \leq  \frac{n}{2}$, hence
\[
\frac{\det \psi_s^* g}{\det g} \geq e^{-ns/2} \quad \text{when } s \geq 0; 
\] 
\[
\frac{\det \psi_s^* g}{\det g} \geq e^{(\sup_M S - n/2)s} \quad \text{when } s <  0; 
\] 
Note that
\[
|\hat{u}|_h(x,s)^2 = | u |_h (\psi_s(x))^2 \left(\frac{\det g(x)}{\det \psi_s^* g(x)} \right)^q.
\]
By the above calculation, $\hat{u}(x, s)$ satisfies
\begin{equation}\label{eqn: growth of u hat}
|\hat{u}|_h(x,s) \leq Ce^{qns/4}(1+r(\psi_s(x)))^d \leq C (1+ r(x))^d e^{(qn/4 + d/2)s} \quad \text{when } s \geq  0;
\end{equation}
and 
\begin{equation}\label{eqn: growth of u hat when s < 0}
|\hat{u}|_h(x,s) \leq  C (1+ r(x))^d e^{(q(n/2 -  \sup_M S) - d/2)s} \quad \text{when } s <  0;
\end{equation}
where we allow the value of the constant $C$ to change when necessary. 

\begin{prop}\label{prop: control the dimension by counting number} On a complete gradient K\"ahler Ricci shrinker with bounded scalar curvature, for any integer $q> 0$, we have 
\[
\dim \mathcal{O}_d(K_M^{-q}) \leq \mathcal{N}\left(\frac{d}{2} + q( \sup_M S - n/2) \right),
\]
where $\mathcal{N}(\cdot)$ is the spectral counting number of the operator $L = - \Delta  + \n_{\n f} - \frac{nq}{4}$ acting on $L^2(dv_f)$ sections of $K_M^{-q}$.
\end{prop}
\begin{proof}
The operator $L $ is self-adjoint w.r.t the measure $dv_f$, and has discrete spectrum as discussed in the previous section. Let $\lambda_1 \leq \lambda_2 \leq ... $ be the eigenvalues of $L$, and let $\phi_1, \phi_2, ...$ be the corresponding eigensections, then $\phi_1, \phi_2, ...$ is a complete orthonormal basis for the space of $L^2(dv_f)$ sections of $K_M^{-q}$. Hence there exist coefficients $a_1, a_2,...,$ such that 
\[
\hat{u}(x, 0) = \sum_{i = 1}^\infty a_i \phi_i(x).
\]
Then the section $\tilde{u}(x, s) : = \sum_{i = 1}^\infty a_i e^{-\lambda_i s} \phi_i(x)$ is an $L^2(dv_f)$ solution of the heat type equation (\ref{eqn: heat type equation for L}) with initial data $\hat{u}(x, 0)$.

Since $\hat{u}$ satisfies (\ref{eqn: growth of u hat}), and that $f$ has quadratic growth, we know that $\hat{u}$ is also an $L^2(dv_f)$ solution of (\ref{eqn: heat type equation for L}), with the same initial data at $s = 0$.

The Cauchy problem for (\ref{eqn: heat type equation for L}) in $L^2(dv_f)$ is uniquely determined by the initial data. To see this, note that the section $\tilde{u} : = e^{\frac{nq}{4} s} u$ satisfies the heat-type equation 
\[
\frac{\pd}{\pd s} \tilde{u} = (\Delta - \n_{\n f}) \tilde{u},
\]
where $\Delta - \n_{\n f}$ is self-adjoint w.r.t the weighted $L^2(dv_f)$-norm. Hence
\[
\frac{d}{ds} \int |\tilde{u}|_h^2 \hspace{5pt} dv_f = - 2 \int |\n \tilde{u}|_h^2 \hspace{5pt} dv_f \leq 0,
\]
where the integration by parts can be verified by cut-off argument. The uniqueness then follows from standard argument.

Consequently, we have $\hat{u}(x, s) = \tilde{u}(x, s)$ for $s \in (-\infty, \infty)$. Now observe that (\ref{eqn: growth of u hat when s < 0}) implies that the $L^2(dv_f)$ norm of $\hat{u}$ satisfies
\[
\|\hat{u}(\cdot, s)\|^2_{L^2(dv_f)} \leq C e^{( q(n - 2 \sup_M S) - d)s} \quad \text{when } s <  0;
\]
while $\tilde{u}$ satisfies
\[
\|\tilde{u}(\cdot, s)\|^2_{L^2(dv_f)} = \sum_{i = 1}^\infty |a_i|^2 e^{-2\lambda_i s}.
\]
Thus the coefficients $a_i$ has to vanish when $2 \lambda_i > d +  q( 2 \sup_M S - n)$, i.e. $\hat{u}(\cdot, 0)$ is actually a finite linear combination of eigensections with eigenvalues less than or equal to $\frac{d}{2} + q( \sup_M S - n/2)$. The claim then follows from the injectivity of the map $u \mapsto \hat{u}$. 

\end{proof}

Now we can prove Corollary \ref{cor: dimension estimate for power of anticanonical bundle} in the introduction.
\begin{proof}[Proof of Corollary \ref{cor: dimension estimate for power of anticanonical bundle}.]
The operator $L = - \Delta  + \n_{\n f} - \frac{nq}{4}$ clearly satisfies the conditions in Theorem \ref{thm: estimate of spectral counting number}, hence the corollary follows from Proposition \ref{prop: control the dimension by counting number} and Theorem \ref{thm: estimate of spectral counting number}.
\end{proof}

\section{Frequency on asymptotically conical gradient K\"ahler Ricci shrinker}\label{section: frequency on AC shrinkers}
In this section we derive a weak monotonicity theorem for the frequency of holomorphic sections of $K_M^{-q}$ on gradient K\"ahler Ricci shrinkers which are asymptotically conical.

\subsection{Definitions and basic properties}
Let $u$ be a holomorphic section of $K_M^{-q}$. In the following calculations, we denote the Hermitian inner product on $K_M^{-q}$ as $h(,)$.  The Hermitian connection is denoted for simplicity as $\n$, it is compatible with the metric $ h$ so that 
\[
X h(u, v) = h(\n_X u, v) + h(u, \n_{\bar{X}} v)
\]
 for any sections $u$ and $v$. Define the pointwise norm of a section as 
 \[
 |u|_h = \sqrt{h(u,u)},
 \]
  and define the norm for its covariant differential as 
  \[
  |\n u|_h = \sqrt{g^{i \jb} h(\n_i u, \n_j u)}.
  \] 
Since $u$ is holomorphic, we have $\n_{\ib} u = 0$ and $\n_{\jb} \n_i u = - q R_{i \jb} u$, and 
\begin{equation}\label{eqn: laplacian of norm square of u}
\Delta |u|_h^2 = 2 g^{i\jb} \n_i \n_{\jb} h(u, u) = 2g^{i \jb}h(\n_i u, \n_j u) +  2 g^{i \jb} h(u, \n_{\ib} \n_j u) = 2 |\n u|_h^2 -  q S |u|_h^2,
\end{equation}
where $S = 2 g^{i \jb} R_{i \jb}$.
 
Let $n = 2m$ be the real dimension of the K\"ahler manifold. Define
\[
I(r) = r^{1-n} \int_{b=r} |u|_h^2 |\n b|,
\]
\[
D(r) = \frac{1}{2} r^{2-n} \int_{b=r} \la \n |u|_h^2, \vn \ra, 
\]
where $\vn = \frac{\n b}{|\n b|}$ is the unit outer normal vector. This is well-defined when $|\n b| > 0$, which must hold when $b > 2 \sqrt{\sup_M S}$ since we have the identity $1 = |\n b|^2 + \frac{4S}{b^2}$, and we assume that $\sup_M S < \infty$.

By the divergence theorem and the soliton identities, we can write $I(r)$ as
\[
I(r) = \int_{b< r} \langle \n |u|_h^2 , \n b \rangle  b^{1-n} + |u|_h^2 b^{-n} \left( \frac{4n}{b^2} - 2 \right) S.
\]
The derivative of $I(r)$ can be calculated as the following,
\begin{equation}\label{eqn: derivative of I}
I'(r) =\frac{2D(r)}{r} + r^{-n} \left( \frac{4n}{r^2} - 2\right) \int_{b=r} \frac{S |u|_h^2}{|\n b|}. 
\end{equation}

By the divergence theorem we can write
\begin{equation}\label{eqn: writting D by divergence theorem}
D(r) = r^{2-n} \int_{b< r} |\n u|_h^2 - \frac{q}{2} S |u|_h^2.
\end{equation}
Whenever $I(r)> 0$, we define the frequency function for $u$ as
\[
U(r) = \frac{D(r)}{I(r)}.
\]

The following lemma is useful in calculating $D'(r)$.
\begin{lem}\label{lem: first variation formula kahler case}
Let $u$ be a holomorphic section of $K_M^{-q}$ on a compact domain $\Omega$ of a K\"ahler manifold, let $\vn$ be the unit outer normal vector of the boundary $\pd \Omega$. For any real vector field $X$, we have
\[
\begin{split}
 \int_{\Omega} |\n u|_h^2 div(X)
= & \int_{\Omega}   2 Re(\la h( \n u, \n u), \n  X \ra) - \frac{q}{2}S \la \n |u|_h^2,  X \ra + q \la Ric(X),  \n |u|_h^2 \ra \\
& - \int_{\pd \Omega} \la h(\n u, \n_{\vn} u),  X \ra  + \la h( \n_{\vn} u, \n u),  X \ra  + \int_{\pd \Omega} |\n u|^2 \langle X, \vn \rangle. \\ 
\end{split}
\] 
\end{lem}
\begin{proof}
We prove this equation using integration by parts. Denote the unit normal vector as 
\[
\vn  = \nu_k g^{k \lb} \frac{\pd}{\pd \bar{z}^l}  + \nu_{\kb} g^{l \kb} \frac{\pd}{\pd {z}^l},
\]
{and $X$ is a real vector field.}
\[
\begin{split}
\int_{\Omega} |\n u|_h^2 div(X) 
= & \int_{\Omega} - g^{i \jb} \pd_i |\n u|_h^2 X_{\jb} - g^{i\jb} \pd_{\jb} |\n u|_h^2 X_i + \int_{\pd \Omega} |\n u|_h^2 \langle X, \vn \rangle \\
= &  \int_{\Omega} -  g^{i \jb} g^{k \lb} [h( \n_i \n_k u, \n_l u) + h(  \n_k u, \n_{\ib} \n_l u)] X_{\jb} \\
& -\int_{\Omega} g^{i \jb} g^{k \lb} [h( \n_{\jb} \n_k u, \n_l u) + h(  \n_k u, \n_{j} \n_l u)] X_i + \int_{\pd \Omega} |\n u|_h^2 \langle X, \vn \rangle, \\
\end{split}
\]
note that $\n_i \n_j u = \n_j \n_i u$ by $R_{ij} = 0$, after integration by parts {for $\int_{\Omega}g^{i \jb} g^{k \lb} h( \n_i \n_k u, \n_l u)X_{\jb}$ and $\int_{\Omega}g^{i \jb} g^{k \lb} h( \n_k u, \n_j\n_l u)X_{i}$}, we have
\[
\begin{split}
& \int_{\Omega} |\n u|_h^2 div(X)\\
= & \int_{\Omega}   g^{i \jb} g^{k \lb} h( \n_i u, \n_l u) \n _k X_{\jb} + g^{i \jb} g^{k \lb} [ h( \n_i u, \n_{\kb} \n_l u)  -h(  \n_k u, \n_{\ib} \n_l u)] X_{\jb} \\
+  & \int_{\Omega}   g^{i \jb} g^{k \lb} h( \n_k u, \n_j u) \n _{\lb} X_{i} + g^{i \jb} g^{k \lb} [ h( \n_{\lb} \n_k u,  \n_j u)  -h(  \n_{\jb} \n_k u, \n_l u) ] X_i\\
& - \int_{\pd \Omega} g^{i \jb} g^{k \lb} h(\n_i u, \n_l u)  X_{\jb} \nu_k + g^{i \jb} g^{k \lb} h( \n_k u, \n_j u)  X_i \nu_{\lb} + \int_{\pd \Omega} |\n u|_h^2 \langle X, \vn \rangle. \\ 
\end{split}
\]
Using $\n_{\jb} \n_i u = - q R_{i\jb} u$, and $\n_{\ib} u = 0$, we have 
\[
\begin{split}
& \int_{\Omega} |\n u|_h^2 div(X)\\
= & \int_{\Omega}   g^{i \jb} g^{k \lb} h( \n_i u, \n_l u) \n _k X_{\jb} - g^{i \jb} g^{k \lb} [ q R_{k \lb} h( \n_i u, u)  - q R_{i \lb} h(  \n_k u, u)] X_{\jb} \\
+  & \int_{\Omega}   g^{i \jb} g^{k \lb} h( \n_k u, \n_j u) \n _{\lb} X_{i} - g^{i \jb} g^{k \lb} [ q R_{k \lb} h( u,  \n_j u)  - q R_{k \jb } h( u, \n_l u) ] X_i\\
& - \int_{\pd \Omega} g^{i \jb} g^{k \lb} h(\n_i u, \n_l u)  X_{\jb} \nu_k + g^{i \jb} g^{k \lb} h( \n_k u, \n_j u)  X_i \nu_{\lb} + \int_{\pd \Omega} |\n u|_h^2 \langle X, \vn \rangle, \\ 
\end{split}
\] 
note that {$S = 2 g^{k \lb} R_{k \lb}$},  this finishes the proof.
\end{proof}
\subsection{Asymptotically conical shrinkers}
In \cite{KW2015}, a gradient Ricci shrinker $(M, g, f)$ is called \textbf{asymptotically conical} if there exists a regular cone $(C:= \Sigma \times (0, \infty), g^c = d\rho^2 + \rho^2 g^\Sigma)$, and a diffeomorphism $\Psi: \Sigma \times (R, \infty) \to M \backslash\Omega$, where $R> 0$ is some constant and $\Omega$ is a compact domain in $M$, such that $\lambda^{-2}\eta_\lambda^* \Psi^* g \to g^c$ in $C^\infty_{loc}$ as $\lambda \to \infty$, where $\eta_\lambda$ is the homothety $\eta_\lambda(x, r) = (x, \lambda r)$ on the cone. The only asymptotically conical gradient Ricci shrinker in real dimension $2$ is the Gaussian soliton. Examples in higher dimensions are given by \cite{FIK2003}. We will assume that the real dimension of $M$ is $n \geq 4$ in this section.

By \cite{KW2015}, a gradient K\"ahler Ricci shrinker is asymptotically conical if and only if it has quadratic curvature decay, i.e. $|Rm|(x) \leq C (1+r(x))^{-2}$ for some constant $C$ and $r(x)$ is the distance to a fixed point; and its asymptotic cone is unique. It's clear that the asymptotic cone $(C, g^c)$ can also be obtained by taking pointed Cheeger-Gromov limit of $(M, p, \lambda^{-2} g)$ for any fixed point $p$ and a sequence of $\lambda \to \infty$.

In the K\"ahler case, the section $\Sigma$ of the asymptotic cone is connected since a gradient K\"ahler Ricci shrinker is connected at infinity by \cite{MW2015b}. In this case the complex structure of the blow-down sequence converges to the complex structure $J^c$ on the cone, and the cone $(C, g^c)$ is K\"ahler, with K\"ahler form given by $\omega^c = \sqrt{-1}\pd \dbar \rho^2$. Define $f^c = \frac{\rho^2}{4}$, then we have 
\[
\n \n f^c = \frac{1}{2} g^c 
\]
on the cone $(C, g^c)$. For any holomorphic section $u$ of $K_{C}^{-q}$, define 
\[
I^c(r) = r^{1-n} \int_{\rho = r} |u|_h^2, \quad D^c(r) = \frac{1}{2}r^{2-n} \int_{\rho = r} \n_{\n \rho}|u|_h^2 , \quad U^c(r) = \frac{D^c(r)}{I^c(r)}.
\]

Note that $|\n \rho | = 1$, $\n\n \rho = \rho g^\Sigma$ and $\Delta \rho = \frac{n-1}{\rho}$. 
The calculation of $I^c(r)'$ becomes simple in this case, and we have 
\[
I^c(r)' = \frac{2D^c(r)}{r}.
\]

{
Let $u$ be a nontrivial holomorphic section, since $u$ is analytic, it $I^c(r)$ cannot vanish on any interval, so there must be a sequence of $r_i \to 0$ such that $I^c(r_i) > 0$, then $I^c(r)$ must be positive for all $r> 0$ since it is nondecreasing. 
}

\begin{lem}\label{lem: monotonicity of D^c}Suppose $|u|_h$ is bounded near the vertex, then we have
\[
D^c(r)' =  r^{2-n} \int_{\rho = r} 2 |\n_{\vn} u|_h^2,
\]
where $\vn = \n \rho$ is the unit outer normal vector. 
\end{lem}
\begin{proof}
By the divergence theorem, we can write $D^c(r)$ as
\[
D^c(r) = \frac{1}{2}r^{2-n} \int_{\rho = r} \n_{\n \rho} |u|^2_h =  \frac{1}{2} r^{2-n} \left( \int_{\epsilon < \rho < r} \Delta |u|_h^2 + \int_{\rho =  \epsilon} \n_{\n \rho} |u|_h^2\right) .
\]
Since $|u|_h$ is bounded and the curvature is $O(r^{-2})$ near the vertex, by the gradient estimate we have $|\n|u|| = O(r^{-1})$, hence
\[
\int_{\rho  = \epsilon} \n_{\n \rho} |u|_h^2 \to 0 \text{ as } \epsilon \to 0
\]
when the real dimension $n > 2$. Therefore by (\ref{eqn: laplacian of norm square of u}) we have
\[
D^c(r) =  \frac{1}{2} r^{2-n} \int_{\rho < r} \Delta |u|_h^2 = r^{2-n} \int_{\rho < r} |\n u|_h^2 - \frac{q}{2} S |u|_h^2,
\]
and then 
\begin{equation}\label{eqn: derivative of D^c -- 1}
D^c(r)' = \frac{2-n}{r} D^c(r) + r^{2-n} \int_{\rho = r}|\n u|_h^2 - \frac{q}{2} S |u|_h^2.
\end{equation}
Similarly we can verify that integration by parts works in the proof of Lemma \ref{lem: first variation formula kahler case}, with $\Omega = \{\rho < r\}$ and $X = \n f^c = \frac{\rho}{2} \n \rho$, which then yields 
\[
\frac{r}{2}\int_{\rho = r} |\n u|_h^2 - r \int_{\rho = r} |\n_{\vn} u|^2 = \frac{n-2}{2} \int_{\rho < r} |\n u|_h^2 + \int_{\rho < r} \frac{q}{2} S \la \n |u|_h^2, \n f^c\ra.
\]
Apply the equation above in (\ref{eqn: derivative of D^c -- 1}) yields
\begin{equation}\label{eqn: derivative of D^c -- 2}
\begin{split}
D^c(r)' = & r^{2-n} \int_{\rho = r} 2 |\n_{\vn} u|_h^2 + 2 r^{1-n} \int_{\rho < r} \frac{q}{2} S \la \n |u|_h^2, \n f^c\ra \\
& -  r^{2-n} \int_{\rho = r} \frac{q}{2} S |u|_h^2 + (n-2) r^{1-n} \int_{\rho < r} \frac{q}{2} S |u|_h^2.
\end{split}
\end{equation}
Using integration by parts in the second term on the RHS (this can be verified by the same method as before), we get 
\[
\int_{\rho < r} S \la \n |u|_h^2, \n f^c\ra = - \frac{n}{2} \int_{\rho < r} S |u|_h^2 - \frac{1}{2} \int_{\rho < r} \rho \frac{\pd S}{ \pd \rho} |u|_h^2 + \frac{r}{2} \int_{\rho = r} S|u|_h^2.
\]
Note that the scalar curvature satisfies $S(x,\rho) = \rho^{-2} S(x, 1)$, hence 
\[
\frac{\pd S}{\pd \rho} = -\frac{2}{\rho } S.
\]
Then we observe that the curvature terms in (\ref{eqn: derivative of D^c -- 2}) cancel out, and we have 
\[
D^c(r)' =  r^{2-n} \int_{\rho = r} 2 |\n_{\vn} u|_h^2 .
\]
\end{proof}

Then by direct calculation and applying Holder's inequality, we get 
\begin{lem}\label{lem: monotonicity of U^c}
Suppose $|u|_h$ is bounded near the vertex, then we have 
\[
U^c(r)' \geq 0,
\]
and the equality holds if and only if $\n_{\n \rho} u = \frac{U(r)}{r} u$.
\end{lem}
\begin{proof}
This inequality comes from Holder's inequality, when equality holds we have $\n_{\n \rho} u = \mu u$ for some constant $\mu$, then direct calculation of $U(r)$ shows that $\mu = \frac{U(r)}{r}$,
\end{proof}

\begin{thm}\label{eqn: weak monotonicity of U in the AC case}
There exists a discrete subset $\mathcal{S} \subset \mathbb{R}$, given any $\delta \in (0, 1)$, and $\alpha \in \mathbb{R}_{\color{blue} +} \backslash \mathcal{S}$, there exists some constant $R_0$, such that for any $r > R_0$, and for any holomorphic section $u$, if $U_{u}(r) < \alpha$, we must have $U_{u}(\theta r) < \alpha$ for all $\theta \in (\delta, 1-\delta)$.
\end{thm}
\begin{proof}
We argue by contradiction. Suppose the theorem is not true, then there is an $\alpha_0$, a sequence of radius $r_i \to \infty$, a sequence $\theta_i \in (\delta, 1-\delta)$, and a sequence of holomorphic sections $u_i$, such that 
\[
U_{u_i}(r_i) < \alpha_0, \text{ while } U_{u_i}(\theta_i r_i) \geq \alpha_0.
\]
By passing to a subsequence, we can assume WLOG that $\theta_i \to \theta \in [\delta, 1-\delta]$.

{
We need to take a limit of this sequence of sections $u_i$ after properly scaling to get a contradiction. 

By the asymptotically conical assumption in the K\"ahler case, for some comapct subset $\Omega \subset M$, $\Psi: C = \Sigma \times (0, \infty) \to M \backslash \Omega$ is a biholomorphism w.r.t. the pullback complex structure. (In fact, there is a holomorphic resolution map $\pi: M \to C$ which is biholomorphic onto the smooth part of $C$ \cite{CDS2024}.) To work with holomorphic sections, it is beneficial to use the biholomorphisms generated by $\n f$ to pullback the metric, rather than using homotheties. Let $\Phi_t$ be the family of biholomorphisms generated by $\n f$. By the quadratic curvature decay, it is easy to compute that $f(\Phi_t) = e^t f + O(f^{-1})$, hence $e^t r^2 \approx r(\Phi_t)^2$. 

Fix $0< \sigma < \delta$ sufficiently small. Take a sequence of $t_i = \ln (r_i^2/4) \to \infty$, let $ g_i : = 4^{-1}e^{-t_i} \Psi^* \Phi_{t_i}^*g$, then $g_i$ converges subsequentially on $C$, $b_i := e^{-t_i/2} \Psi^*\Phi_{t_i}^* b$ converges subsequentially to $\rho^c$, and the convergence is smooth and uniform on $\Sigma \times [a, b]$ for any compact interval $[a, b]$ where $a> 0$.

Define $v_i = \frac{u_i}{\sqrt{I_{u_i}(\frac{r_i}{2})}}$, then we have $I_{v_i}(\frac{r_i}{2})=1$.
Let $\phi_i = \Phi_{t_i}\circ \Psi$, and let $\tilde{v}_i = \phi_i^{*} v_i$, then $\tilde{v}_i$ is a holomorphic section on the cone which is bounded near the vertex. We need to show that $\tilde{v}_i$ converges to a nontrivial holomorphic section on the annulus region $\theta < \rho^c < 1$. 

First we need to show that $I^c_{\tilde{v}_i}(r) > 0$, for all $ \theta \leq r \leq 1 $ when $i$ is large enough. To do this, let's take a sequence of rescaling of $\tilde{v}_i$, let $\hat{v}_i = K_i \tilde{v}_i$ for some factors $K_i$ such that $I_{g_i, \hat{v}_i}(\theta) = 1$, it sufficies to show that $I^c_{\hat{v}_i}(r)> 0$.  Because $\{b_i = 1/2\}=\phi_i^{-1}\{b=r_i/2\}$ converges to $\{\rho^c=1/2\}\subset C$, and the metrics converge uniformly on compact domains, for any $\epsilon > 0$, there exists some $i_0$ large enough, such that when $i > i_0$, there are diffeomrphisms $\psi_i: \{\rho^c = r\} \to \{b_i = r\}$, and $(1-\epsilon)g^c|_{\rho^c = r}< \psi_i^*g_i|_{b_i = r} < (1+\epsilon)g^c|_{\rho^c = r}$. Then we have 
\[
(1-\epsilon)I_{g_i, \hat{v}_i}(r)< I^c_{ \psi_i^*\hat{v}_i}(r) < (1+\epsilon)I_{g_i, \hat{v}_i}(r),
\]
whenever $I_{g_i, \hat{v}_i}(r) > 0$, for possibly re-choice of  $\epsilon> 0$. 
In particular this implies $1-\epsilon< I^c_{\hat{v}_i}(\theta)< 1+\epsilon$ by continuity. Similarly we can compare $D^c_{\hat{v}_i}(\theta)$ with $D_{g_i, \hat{v}_i}(\theta)$, and we have $U^c_{\hat{v}_i}(\theta) > (1-\epsilon) \alpha_0$ for some $\epsilon > 0$. In particular, we have $D^c_{\hat{v}_i}(\theta) > 0$, hence by Lemma \ref{lem: monotonicity of D^c}, we have $D^c_{\hat{v}_i}(r) > 0$ for all $r > \theta$, then $I^c_{\hat{v}_i}(r)$ is increasing for $r > \theta$, hence it stays positive. Consequently, since scaling the section does not change the sign of $I^c$ and $D^c$, we have $I^c_{\tilde{v}_i}(r)> 0$ and $D^c_{\tilde{v}_i}(r)> 0$ for $r \geq \theta$, so $U^c_{\tilde{v}_i}(r)$ is well-defined for $r\geq \theta$. By similar argument, we can see that $U^c_{\tilde{v}_i}(1) < (1+\epsilon) \alpha_0$
for some $\epsilon > 0$ when $i$ is large enough. Then by Lemma \ref{lem: monotonicity of U^c} we have
\begin{equation}\label{eqn: U^c upper and lower bound}
(1-\epsilon)\alpha_0< U^c_{\tilde{v}_i}(r) \leq (1+\epsilon) \alpha_0,
\end{equation}
for all $\theta \leq r \leq 1 $, which implies that $I^c_{\tilde{v}_i}(r) > 0$ for all $\theta \leq r \leq 1 $.



Then we can show that $I^c_{\tilde{v}_i}(r)$ is uniformly controlled when $r\in[\theta, 1]$.
By argument similar to the above we have $1-\epsilon < I^c_{\tilde{v}_i}(1/2) < 1+\epsilon$. Then by $(\ln I^c_{ \tilde{v}_i}(r))' = \frac{2 U^c_{\tilde{v}_i}(r)}{r}$, which is now nonnegative and bounded from above, we have
\[
I^c_{ \tilde{v}_i}(r) \leq I^c_{ \tilde{v}_i}(1/2)\quad \text{for }  \theta \leq r \leq \frac{1}{2};
\]
and
\[
I^c_{\tilde{v}_i}(r) \leq I^c_{\tilde{v}_i}(1/2) (2r)^{2(1+\epsilon)\alpha_0}, \quad \text{for } \frac{1}{2} \leq r \leq 1.
\]
After integrating in the variable $r$, we have uniform $L^2$ control of $\tilde{v}_i$ on any annulus region contained in $\{\theta \leq  \rho^c \leq 1\}$, then using the mean value inequality locally we obtain uniform $L^\infty$ control in the annulus domain $\{\sigma\leq \rho \leq 1\}$, and uniform derivative estimates follow from standard theory. So $\tilde{v}_i$ subsequentially converges uniformly in the annulus $\{ \theta +\epsilon \leq \rho^c \leq 1-\epsilon \}$ to a limit $\tilde{v}_\infty$, satisfying
\[
(1-\epsilon) < I^c_{ \tilde{v}_\infty}(1/2)< (1+\epsilon),
\]
hence it is a nontrivial holomorphic section. Similarly as above, we can show uniform positive lower estimates 
\[
I^c_{\tilde{v}_\infty}(r) \geq (1-\epsilon)(2r)^{-(1+\epsilon)\alpha_0} \quad \text{ for } \theta+\epsilon \leq r \leq \frac{1}{2},
\]
and
\[
I^c_{\tilde{v}_\infty}(r) \geq I^c_{\tilde{v}_\infty}(\frac{1}{2}) > 1-\epsilon \quad \text{for } \frac{1}{2} \leq r \leq 1-\epsilon.
\]
Therefore we also have convergence of the frequency,
\[
(1-\epsilon)\alpha_0 \leq U^c_{ \tilde{v}_\infty}(r) \leq (1+\epsilon) \alpha_0, \quad \theta + \epsilon < r < 1-\epsilon. 
\]
Since $\epsilon$ is arbitrary, we can let $\epsilon \to 0$ and extract a diagonal subsequence of $\tilde{v}_i$, which converges to a limit $v_\infty$ on $\{\theta < \rho^c< 1\}$, such that 
\[
U^c_{{v}_\infty}(r) = \alpha_0, \quad \theta < r < 1.
\]
}
By Lemma \ref{lem: monotonicity of U^c} we have 
\[
\n_{\n \rho} v_\infty = \frac{\alpha_0}{\rho} v_\infty, \quad \theta \leq \rho \leq 1. 
\]
We derive from the above equation that 
\[
v_\infty(x, \rho ) = \rho^{\alpha_0} v_\infty(x, 1), \quad \theta \leq \rho \leq 1. 
\]
Let $\Sigma$ denote the $\rho = 1$ cross section of the cone with the induced metric. Write $g^c = d\rho^2 + \rho^2 g^\Sigma$, the hermitian metric $h$ on $K_{C}^{-q}$ can be locally written as $h(u,v) = u\bar{v} e^{-\psi}$ for any sections locally represented by $u,v$, where
\[
\psi =  -\frac{q}{2} \ln \det g^c = - \frac{q}{2} \ln \det g^\Sigma -q(n-1) \ln \rho.
\]

Let $X = J^c(\n \rho)$, it is well-known that $X$ is a unit Killing vector field on the cone sections with constant $\rho$. For simplicity we denote $\tilde{X} = X(\cdot, 1)$, we have $X (\cdot, \rho)= \rho^{-1} \tilde{X}$. Let $e_1 = \frac{1}{\sqrt{2}} (\n \rho - \sqrt{-1} X)$, $\bar{e}_1 = \frac{1}{\sqrt{2}} (\n \rho + \sqrt{-1} X)$. Note that both $\n \rho$ and $X$ are geodesic vector fields, hence by direct calculation we have $\n_{e_1} \bar{e}_1 + \n_{\bar{e}_1} e_1 = 0$.

Complete $e_1$ into a unitary frame by joining $e_2, e_3 ..., e_n,$ tangent to the cone sections, $e_i = \rho^{-1} E_i$, $i = 2, 3, ..., n$, where $\{E_i\}_{i=2}^n$ are unitary on $\Sigma$. The connection $(1,0)$-form is given by $\theta =  - \pd \psi$. By direct computation we have 
\[
\theta(e_1) = -\frac{1}{\sqrt{2}} (\n \rho - \sqrt{-1} X)\psi = \frac{q(n-1)}{\sqrt{2}} \frac{1}{\rho}-\frac{q\sqrt{-1}}{2\sqrt{2}} \rho^{-1} \tilde{X} (\ln \det g^\Sigma).
\]

For the holomorphic section $v_\infty$, we can compute in the unitary frame above to get 
\[
\begin{split}
& \left( \n_{e_1}\n_{\bar{e}_1} + \n_{\bar{e}_1}\n_{e_1} \right) v_\infty\\
= & e_1 (\bar{e}_1 v_\infty) + \bar{e}_1 (e_1 v_\infty) + \bar{e}_1(\theta(e_1)) v_\infty  \\
= & \frac{\pd^2}{\pd \rho^2} v_\infty + X (X v_\infty) + \frac{1}{\sqrt{2}} \frac{\pd}{\pd \rho} \theta(e_1) v_\infty + \frac{\sqrt{-1}}{\sqrt{2}} \rho^{-1} \tilde{X}(\theta(e_1)) v_\infty\\
= &  \rho^{\alpha_0 - 2} \tilde{X}(\tilde{X} v_\infty(\cdot, 1)) \\
& + \left(\alpha_0 (\alpha_0 - 1 ) - \frac{q(n-1)}{2} + \frac{q\sqrt{-1}}{4} \tilde{X} (\ln \det g^\Sigma) + \frac{q}{4} \tilde{X}\tilde{X}(\ln \det g^\Sigma)\right)\rho^{\alpha_0 - 2} v_\infty(\cdot, 1).
\end{split}\]
Hence 
\[
\begin{split}
 \Delta_{g^c} v_\infty = &\sum_{i = 1}^n \left( \n_{e_i }\n_{\bar{e}_i} + \n_{\bar{e}_i}\n_{e_i} \right) v_\infty \\
 = & \left(\alpha_0 (\alpha_0 - 1 ) - \frac{q(n-1)}{2}\right) \rho^{\alpha_0 - 2} v_\infty(\cdot , 1) +   \rho^{\alpha_0 - 2} L^\Sigma v_\infty(\cdot, 1),
\end{split}
\]
where 
\[ L^\Sigma : = \tilde{X} \tilde{X} + \sum_{i = 2}^n \left( \n_{E_i} \n_{\bar{E}_i} + \n_{\bar{E}_i} \n_{E_i} \right) + \frac{q\sqrt{-1}}{4} \tilde{X} (\ln \det g^\Sigma) + \frac{q}{4} \tilde{X}\tilde{X}(\ln \det g^\Sigma)
\]
is an elliptic operator of second order on sections of the vector bundle $K_C^{-q}$ restricted to $\Sigma$. 

On the other hand, by (\ref{eqn: Laplacian of a holomorphic section}) we have 
\[
\Delta_{g^c} v_\infty = - \frac{q}{2} S_{g^c} v_\infty = -\frac{q}{2} \rho^{\alpha_0-2} S_{g^c}(\cdot, 1) v_\infty(\cdot, 1).
\]
Therefore we must have 
\[
 \left( L^\Sigma +\frac{q}{2} S_{g^c} \right) v_{\infty}(x, 1) =- \left(\alpha_0 (\alpha_0 - 1 ) - \frac{q(n-1)}{2}\right) v_\infty(x, 1),
\]
i.e. the number $- \alpha_0(\alpha_0 - 1) + \frac{1}{2}(n-1) q$ has to be an eigenvalue of the operator $L^\Sigma+\frac{q}{2} S_{g^c}(\cdot, 1)$. By the compactness of $\Sigma$, the spectrum of this operator is discrete, this finishes the proof. 
\end{proof}

\section{Existence of holomorphic sections with polynomial growth}\label{section: existence}
In this section, let $(M, g, f)$ be a gradient K\"ahler Ricci shrinker which is asymptotically conical, keep the notations of the previous section, and denote $D_r=\{b < r\}$.  We construct holomorphic sections of $K_M^{-q}$ with polynomial growth. Then as an application we show that the Kodaira map constructed by polynomial growth sections of $K_M^{-q}$ is a holomorphic embedding.

\begin{lem}\label{lem: local construction of holomorphic sections}
There exists a constant $\bar{q}$ depending on the geometry, such that for any $x \neq y$ in $D_{\frac{r_0}{5}} = \{b < \frac{r_0}{5}\}$, where $r_0> 10$ is any fixed number, and for any given $K_M^{-q}$-valued $(1, 0)$-form $\alpha$ at $x$, there exists a holomorphic section $u_0$ of $K_M^{-q}$ for $q \geq \bar{q}$, such that $u_0(x) \neq 0$, $u_0(y) = 0$, $\pd u_0(x) = \alpha$, moreover,
\[
\sup_{b < r_0} |u_0|_h^2 \leq C, \quad and \quad \sup_{ b =r} |u_0|_h^2 \geq 1 \quad for \quad r \in [ \frac{r_0}{4} , r_0],
\]
where $C$ is a constant depending on the geometry and $n, q, r_0$.
\end{lem}
\begin{proof}
Since $(M, g)$ is noncollapsed (\cite{CN2009}) with bounded curvature, there exists a $\delta_0>0$, such that for any point $x \in M$, there is a neighbourhood $B(x, \delta_0)$ covered by holomorphic coordinates $w^1, w^2, ..., w^m$, in which the metric $g$ is equivalent to the Euclidean metric $\frac{1}{2}\delta_{ij} < g_{i\jb} < 2 \delta_{ij}$.

Let $\eta$ be a smooth function on the real line, such that $\eta = 1$ on $(-\infty,1]$, $\eta = 0$ on $[2, \infty)$, $-2 < \eta' \leq 0$ and $|\eta''|\leq 4 $. 
Define $\xi_x = n \eta\left( 2 \delta_0^{-2} \sum_{i = 1}^m |w^i|^2 \right)\ln \left( \sum_{i = 1}^m |w^i|^2 \right)$. Then $\xi_x$ is compactly supported, $e^{-\xi_x}$ is nonintegrable around $x$, and one can check by direct calculation that $\sqrt{-1}\pd \dbar \xi_x \geq - C(n, \delta_0) \w$, where $\w$ is the K\"ahler form. For any other point $y \neq x$, construct $\xi_y$ in the same way. 

For any point $z \in D_{r_0}\backslash D_{\frac{r_0}{4}}$, choose holomorphic coordinates $w^1, ..., w^m$ in $B(z, \delta_0)$, such that the complex hypersurface $\{w^1 = 0\}$ is transversal to the level set of the potential function $f$ at $z$. Let $H_{z} = \{w^1 = 0\} \cap B(z, \frac{\delta_0}{12})$. By compactness we can find finitely many points $z_1, ..., z_k$, such that the values of $b$ on the union $\cup_{i=1}^k H_{z_i}$ cover the interval $[\frac{r_0}{2}, r_0]$. Denote the coordinate functions centered at $z_i$ as $w^1_i, ..., w^m_i$. Define $\xi_H = \sum_{i=1}^k n \eta(32\delta_0^{-2} \sum_{j=1}^m |w^j_i|^2) \ln |w^1_i|^2$. By direct calculation we have $\sqrt{-1}\pd \dbar \xi_H \geq -C(n, \delta_0, k) \w$.

By the soliton identity $Ric + \n\n f = \frac{1}{2} g$, the metric $h + q f + \xi_x+\xi_y + \xi_H$ is uniformly positive when $q > \bar{q}:=4 C(n, \delta_0) + 2C(n, \delta_0, k) $. Now we fix $q > \bar{q}$. Since $K_M^{-q}$ is trivialized when restricted to $B(x, \delta_0)$, we can choose a nonzero constant section $\tilde{v}$ on $B(x, \delta_0/2)$, with $|\tilde{v}|_h(x) = 1$. Let $\delta = \min \{ 1, d(x, y) \}$, choose a cut-off function $\phi$ supported on $B(x, \delta)$ that equals $1$ in $B(x, \delta/2)$, and $|\n \phi| \leq 4 \delta^{-1}$. Similarly choose locally constant sections $\tilde{v}_i$ and cutoff functions $\phi_i$ on $B(z_i, \delta_0)$ for each $i = 1,..., k$, such that $|\tilde{v}_i|_h(z_i) = 2^i$ and $|\n \phi_i| \leq 4 \delta_0^{-1}$.  Suppose $\alpha = \sum_{i =1}^m a_i dw^i$ at $x$, where the $a_i$ are local sections. Define 
\[
v = \phi\tilde{v}  -  \sum_{i=1}^m \phi a_i w^i + \sum_{i = 1}^k \phi_i \tilde{v}_i.
\] 
Then $v$ is nontrivial and $v$ vanishes outside the union of $B(x, \delta)\cup_{i = 1}^k B(z_i, \delta_0)$,  and $\dbar v = 0$ on $B(x, \delta/2)$. 

Note that $f_{i \jb} > \frac{1}{4} g_{i \jb}$ outside of a compact set of $M$ since the Ricci curvature has quadratic decay, hence we can modify $f$ by a suitable constant in a bounded domain to show that $M$ is weakly pseudoconvex in the sense of \cite{Dem2012} (Section 5). By Corollary 5.3 in \cite{Dem2012}, there exists a section $u$ of $K_M^{-q}$ such that $\dbar u = \dbar v$, and 
\begin{equation}\label{eqn: L2 estimate}
\int_{M} |u|_h^2 e^{-qf - \xi_x - \xi_y-\xi_H} \leq C(n, q, \delta_0,k) \int_{M} |\dbar v|_h^2 e^{-qf - \xi_x - \xi_y-\xi_H}.
\end{equation}
The integrability of the RHS above is guaranteed since $v$ is chosen to be compactly supported and locally constant at the singular points of the weight function. Observe that the integral on the RHS can be calculated on the support of each $\phi, \phi_i$, the gradient $|\n \phi_i|$ is bounded by $4\delta_0^{-1}$ which is a fixed constant; the gradient $|\n \phi|$ may depend on the distance between $x$ and $y$, but on the support of $\phi$ we have 
\[
\int_{B(x, \delta)} |\n \phi|^2 |\tilde{v}|_h^2 e^{-q f - \xi_x - \xi_y -\xi_H} \leq C \delta^{n-2} e^{- q f(x)},
\]
which is bounded when $n\geq 2$, hence we can verify that the right hand side of (\ref{eqn: L2 estimate}) is bounded from above by a constant depending only on $n, q, \delta_0, k$, the constant $k$ is determined by the geometry of the domain $D_{r_0}$, therefore the weighted $L^2$ integral of $u$ is bounded above by a constant independent of $x$ and $y$:
\begin{equation}\label{eqn: integral upper estimate for u}
\int_{M} |u|_h^2 e^{-qf - \xi_x - \xi_y - \xi_H} \leq C(n, q, \delta_0, k) .
\end{equation}
Let $u_0 = u - v$, then $u_0$ is a holomorphic section. Note that the nonintegrability of $e^{-\xi_x}, e^{-\xi_y}$ implies $\pd u(x) = u(x) = u(y) = 0$, hence $u_0(x) = - \tilde{v}(x) \neq 0$, while $u_0(y) = 0$, and $\pd u_0(x) = \sum_{i=1}^ma_i dw^i$. 

 By (\ref{eqn: laplacian of norm square of u}), and by the boundedness of the curvature, we can use Nash-Moser iteration to derive a pointwise estimate for $u_0$ from the integral bound (\ref{eqn: integral upper estimate for u}),
\[
\sup_{b < r_0} |u_0|_h^2 \leq C,
\] 
where $C$ depends on the geometry and on $n, q, \delta_0, r_0$. The nonintegrability of $e^{-\xi_H}$ implies that $u = 0$ on $H_{z_i}$, $i = 1,..., k$, hence for any $z \in \cup H_{z_i}$, we have $u_0(z) = - \sum \phi_i(z) \tilde{v}_i(z)$. Then by the choice of $\tilde{v}_i$ we must have $|u_0(z)|_h \geq 1$. Since the range of $b$ restricted to $\cup H_{z_i}$ covers $[\frac{r_0}{4}, r_0]$, we have verified that
\[
\sup_{b = r} |u_0|_h \geq   1,
\]
for all $r \in [\frac{r_0}{4}, r_0]$.
\end{proof}

The weighted $L^2$ section obtained above may not have polynomial growth, actually we only take its restriction on the bounded domain $D_{r_0}$. The idea is to pull it back to larger domains by the biholomorphism generated by $-\n f$. Now we restate Theorem \ref{thm: existence of polynomial growth holomorphic sections - introduction} in more details and give the proof. 

\begin{thm}\label{thm: existence of polynomial growth holomorphic sections}
Let $(M, g, f)$ be an asymptotically conical gradient K\"ahler Ricci shrinker. There exists a constant $\bar{q}$, for any given $q\geq \bar{q}$, there is a constant $\Lambda$ depending on $n, q$ and the geometry, such that for any points $\bar{x} \neq \bar{y}$, and any $(1,0)$-form $\alpha$ at $x$, there is a holomorphic section $u$ in $K_M^{-q}$, with polynomial growth of order at most $\Lambda$, satisfying $u(\bar{x}) \neq 0,  u(\bar{y}) = 0$ and $\pd u(\bar{x}) = \alpha$.
\end{thm}
\begin{proof}
 Let $\psi_t$ be a family of biholomorphisms generated by 
\[
\pd_t \psi_t = \n f \circ \psi_t, \quad \psi_0 = id. 
\]
Choose $r_0$ large enough depending on the geometry, such that the critical set of $f$ (which is bounded since we have assumed bounded curvature) is contained in $D_{\frac{r_0}{8}}$. Then for any point $x$, we have $\psi^{-1}_t(x) \in D_{\frac{r_0}{5}}$ when $t$ is sufficiently large. 
Take a sequence of positive numbers $t_1 <  t_2 < t_3 ...$, such that $t_i \to \infty$. For each $i$ sufficiently large, let $u_{0, i}$ be the holomorphic section constructed in Lemma \ref{lem: local construction of holomorphic sections}, with $u_{0, i}(\psi_{t_i}^{-1}( \bar{x} )) \neq 0$, $u_{0, i}(\psi_{t_i}^{-1}( \bar{y} )) = 0$, and $\pd u_{0, i}(\psi_{t_i}^{-1}( \bar{x} )) = \psi_{t_i}^{*} \alpha$. Define $u_i = u_{0, i}(\psi_{t_i}^{-1})$, then $u_i( \bar{y} ) = 0$ and $u_i(\bar{x})\neq 0$, moreover $u_i(\bar{x})$ and $\pd u_i(\bar{x})$ are independent of $i$.

Take $r_0$ larger if necessary such that 
$r_0^2 \geq 100 \sup_M S$, then $\sqrt{r_0^2 - 4 \sup_M S} > \frac{1}{2} r_0$. Define 
$$r_i = \frac{1}{2} e^{t_i/2} r_0.$$
We want to derive information of $u_i$ on $D_{r_i}$, from information of $u_0$ on $D_{r_0}$. Thus we need to compute the evolution of various quantities along the flow lines of $\psi_t$.

Observe that $\pd_t \det g(\psi_t) =  \Delta f \det g=(m - S) \det g(\psi_t) $, hence we have 
\begin{equation}\label{eqn: evolution of det(g)}
\det g(x) = e^{mt - k(x, t)} \det g(\psi_t^{-1}(x)),
\end{equation}
where
\[
k(x, t) : = \int_0^t S(\psi_\tau^{-1}(x)) d\tau.
\]
The function $b(\psi_t)$ evolves along the flow lines of $\n f$ by the ODE
\begin{equation}\label{eqn: ode for the evolution of b}
\frac{d}{dt} b(\psi_t(x)) = \frac{1}{2} b(\psi_t(x)) - \frac{2S(\psi_t)}{b(\psi_t(x))}.
\end{equation}
Integrating the above equation yields
\begin{equation}\label{eqn: estimate for b along flow lines}
e^{\frac{t}{2} } \left( b(x) - 2 \sqrt{ \sup_M S} \right)\leq b(\psi_t(x)) \leq e^{\frac{t}{2} } b(x).
\end{equation}
Define $\Sigma(r, t) = \psi_t^{-1} (\{ b = r\})$. By (\ref{eqn: estimate for b along flow lines}), $\Sigma(r, t) \subset D_{e^{-t/2}r + 2\sqrt{\sup_M S}} \backslash D_{e^{-t/2}r}$.

The family of hypersurfaces $\Sigma(r, t)$ evolve{\color{blue}s} by 
\[
\frac{d}{dt} \Sigma(r, t) = - \n f. 
\]
Let $X_1, X_2, ..., X_{n-1}$ be a set of tangent frame on $\Sigma(r, t)$, let $g^\Sigma$ be the induced metric on $\Sigma(r, t)$, then we have 
\[
\begin{split}
\frac{d}{dt} g^{\Sigma}(X_\alpha, X_\beta) = & g(\n_{X_\alpha} (-\n f), X_\beta) + g(X_\alpha, \n_{X_\beta} (-\n f))  = - 2 f_{\alpha \beta}\\
\end{split}
\]
where $\vn$ is the unit normal vector of $\Sigma(r, t)$
. The area form on $\Sigma(r, t)$ evolves by 
\[
\frac{d}{dt} d\sigma|_{\Sigma_{r, t}} = (- tr_{g^\Sigma} \n\n f)  d\sigma|_{\Sigma_{r, t}}= \left(-\frac{n-1}{2} + S - Ric(\vn, \vn) \right)d\sigma|_{\Sigma_{r, t}}.
\]
Hence
\[
d\sigma|_{\Sigma_{r, t}} = e^{- \frac{n-1}{2}t + l(\cdot, t)} d\sigma|_{\Sigma_{r, 0}},
\]
where $l(\cdot , t)$ is defined as 
\[
l(x, t) = \int_0^t (S - Ric(\vn, \vn))(\psi^{-1}_\tau(x)) d\tau. 
\]
By the quadratic curvature decay and (\ref{eqn: estimate for b along flow lines}), we have 
\[
|k(x, t)|, |l(x, t)| \leq C e^t b(x)^{-2}.
\]
The evolution of the hermitian metric $h$ along flow lines of $\psi_t$ can be calculated using (\ref{eqn: evolution of det(g)}). Similarly we can computed the evolution of the metric tensor $g$ along flow lines of $\psi_t$, and we have
\[
g|_{\psi_t^{-1}(x)} = e^{-t + 2 \int_0^t Ric} g|_x,
\]
which is needed to control the evolution of the norm $|\n u|_h(\psi_t^{-1})$ for a section $u$. 

\textbf{Claim:} There exists a constant $\Lambda$ depending only on $n, q, r_0$ and the geometry, such that 
\[U_{u_i}(r_i) < \Lambda.\]

To prove the claim we start by rewritting $I_{u_i}(r)$ and $D_{u_i}(r)$. By a change of variable $y = \psi_t^{-1}(x)$, we can write
\[
\begin{split}
I_{u_i}(r_i) =&  r_i^{1-n} \int_{b = r_i} |u_i(x)|^2_{h(x)} |\n b(x)| d\sigma(x) \\
= &  r_i^{1-n} \int_{\Sigma(r_i, t_i)} |u_0(y)|^2_{h(\psi_{t_i}(y))} |\n b(\psi_{t_i}(y))| d\sigma(\psi_{t_i}(y))\\
=& r_i^{1-n} \int_{\Sigma(r_i, t_i)}e^{(qm+\frac{n-1}{2})t_i - q k(x, t_i) + l(x, t_i)} |u_0(y)|^2_h(y) |\n b(x)| d\sigma(y),
\end{split}
\]
and $D_{u_i}(r_i)$ can be rewritten similarly, thus we have 
\begin{equation}\label{eqn: calculate U(r_i)}
\begin{split}
U_{u_i}(r_i) = \frac{D_{u_i}(r_i)}{I_{u_i}(r_i)} \leq & C(n)r_0\frac{e^{-q \inf k(\cdot, t_i) +\sup l(\cdot, t_i) + 2 \sup \int |Ric| } \int_{\Sigma(r_i, t_i) } |\n u_0|_h |u_0|_h}{e^{ - q \sup k(\cdot, t_i) + \inf l (\cdot, t_i)} \int_{\Sigma(r_i, t_i)} |u_0|^2_h} \\
\leq & C \frac{\int_{\Sigma(r_i, t_i) } |\n u_0|_h |u_0|_h}{\int_{\Sigma(r_i, t_i)} |u_0|^2_h},
\end{split}
\end{equation}
where the sup and inf are taken on $\{b = r_i\}$, these terms are bounded by the quadratic curvature decay,
and we have assumed WLOG that $r_i$ are large enough so that $|\n b|> \frac{1}{2}$, the constant $C$ in the last line depends on $n, r_0, q$ and the curvature bounds. 

Observe that $\Sigma(r_i, t_i) \subset \{\frac{r_0}{2} \leq b \leq \frac{r_0}{2}  + 2\sqrt{\sup_M S} \}$, this is a fixed domain where the curvature is bounded, so we can control $|\n u_0|_h$ by gradient estimate, together with H\"older's inequality we can control the numerator in (\ref{eqn: calculate U(r_i)}) from above. Then we only need a lower bound for the denominator. 

By (\ref{eqn: estimate for b along flow lines}) and the choice of $r_0$, we have 
$\{b = \frac{r_0}{2} \} \subset \cup_{r_i/2 < r < r_i} \Sigma(r, t_i)$.
And we can choose $r_0$ large enough, so that for any point $x\in \{b = \frac{r_0}{2}\}$, the unit ball $B_x(1)$ is contained in the set $\cup_{\frac{r_i}{4} < r < \frac{5}{4}r_i} \Sigma(r, t_i)$.
Since the manifolds has locally bounded geometry, we can apply the Nash-Moser iteration method to equation (\ref{eqn: laplacian of norm square of u}) on unit balls to derive a mean value inequality for $|u_0|_h$, 
\[
\sup_{b = \frac{r_0}{2}} |u_0|_h \leq C \int_{\frac{r_i}{4}} ^{ \frac{5}{4}r_i } \int_{\Sigma(r, t_i)} |u_0|^2_h |\n f|^{-1} d\sigma dr \leq C \int_{\Sigma(r'_i, t_i)} |u_0|^2_h d\sigma,
\]
where the second inequality comes from the intermediate value theorem,
the constant $r'_i$ is in the interval $[\frac{1}{4}r_i, \frac{5}{4}r_i]$. 
By the construction of $u_0$ in Lemma \ref{lem: local construction of holomorphic sections}, $sup_{D_r} |u_0|_h \geq 1$ for all $r \in [ \frac{r_0}{4}, r_0]$, hence 
\[
\int_{\Sigma(r'_i, t_i)} |u_0|_h^2 d\sigma \geq c,
\]
where $c> 0$ is a constant depending only on the geometry and the choice of $q, r_0$.
Note that the sequence ${t_i}$ was arbitrarily chosen, so we can adjust the values of $t_i$ (hence of $r_i$), such that 
\[
\int_{\Sigma(r_i, t_i)} |u_0|_h^2 d\sigma \geq c.
\]
Therefore we get 
\[
U_{u_i}(r_i) < \Lambda,
\]
where $\Lambda$ is a constant depending only on $n, q, r_0$ and the geometry. This proves the claim.

Then by Theorem \ref{eqn: weak monotonicity of U in the AC case}, we have $U_{u_i}(r) \leq \Lambda$ for all $R_0 < r < (1-\delta)r_i$ for some fixed constants $\delta \in (0, 1)$ and $R_0$. Hence by integrating the equation 
\[
I'(r) = \frac{2U(r) I(r)}{r} + r^{-n} \left( \frac{4n}{r^2} - 2\right) \int_{b=r} \frac{S |u|_h^2}{|\n b|}, 
\]
and using the quadratic curvature decay condition, we get that $I_{u_i}(r) \leq C(1+r)^{2\Lambda}$, therefore $|u_i|_h$ must have polynomial growth with order bounded uniformly for all $i$. Note that $|u_i(\bar{x})| > 0$ is a fixed value, so the limit exist{\color{blue}s} when $i \to \infty$, which is the desired holomorphic section. 
\end{proof}

Let $\mathcal{O}_{\Lambda}(K_M^{-q})$ be the linear space of holomorphic sections of $K_M^{-q}$ with polynomial growth of order at most $\Lambda$. By Corollary \ref{cor: dimension estimate for power of anticanonical bundle}, the dimension of $\mathcal{O}_{\Lambda}(K_M^{-q})$ is finite. Let $u_0, u_1, ..., u_N$ be a basis for $\mathcal{O}_{\Lambda}(K_M^{-q})$, the Kodaira map $\mathcal{K}: M \to \mathbb{CP}^N$ is defined by
\[
\mathcal{K}(x) = [u_0(x), u_1(x), ..., u_N(x)]. 
\] 
Now we restate and prove Corollary \ref{cor: Kodaira emdedding of AC shrinkers} in the introduction. 
\begin{cor}
Let $(M, g, f)$ be a gradient K\"ahler Ricci shrinker which is asymptotically conical. Then for $q$ large enough, there are constants $\Lambda$ and $N$, such that the Kodaira map $\mathcal{K}: M \to \mathbb{CP}^N$ constructed by sections of $\mathcal{O}_{\Lambda}(K_M^{-q})$ is a holomorphic embedding. 
\end{cor}
\begin{proof}
By Theorem \ref{thm: existence of polynomial growth holomorphic sections}, when $q$ is sufficiently large, the space $\mathcal{O}_{\Lambda}(K_M^{-q})$ separates points and tangents, and by Corollary \ref{cor: dimension estimate for power of anticanonical bundle} it is finite dimensional, so the proof follows from classic arguments (see for example \cite{Ber2009}). 
\end{proof}

\textbf{Acknowledgement:} We would like to thank Professor Huai-Dong Cao, Professor Jiaping Wang for their interest in this work and helpful conversations. F.H. is partially supported by NSFC grant No.12141101. J.O. is partially supported by NSFC grant No.12401074, 12371061, 12371081.



\end{document}